\documentclass[11pt,leqno]{article}

\usepackage{fullpage}
\usepackage{lscape}
\usepackage{bigints}
\usepackage{framed}
\usepackage{mdframed}
\usepackage{enumerate}
\usepackage{paralist}
\usepackage[T1]{fontenc}
\usepackage{moresize}
\usepackage{bm}
\usepackage{bbm}
\usepackage{dsfont}
\usepackage{amsmath}
\usepackage{amssymb}
\usepackage{amsthm}
\usepackage{amsfonts}
\usepackage{stmaryrd}
\usepackage{array}
\usepackage{mathrsfs}
\usepackage{mathtools} 
\usepackage{extarrows}
\usepackage{stackrel}
\usepackage{relsize,exscale}
\usepackage{scalerel}
\usepackage[nodisplayskipstretch]{setspace}
\usepackage{color}
\usepackage[usenames, dvipsnames]{xcolor}
\usepackage{cancel}
\usepackage{soul}
\usepackage{xfrac}
\usepackage{siunitx}
\usepackage{graphicx}
\usepackage{float}
\usepackage{rotating}
\usepackage{subcaption}
\usepackage{overpic}
\usepackage[all]{xy}
\usepackage{tikz}
\usetikzlibrary{arrows, matrix, positioning, calc, automata, patterns}
\usepackage{booktabs}
\usepackage{dcolumn}
\usepackage{multirow}
\usepackage{diagbox}
\usepackage{tabularx}
\usepackage{verbatim}
\usepackage{listings}
\usepackage[ruled,vlined]{algorithm2e}
\usepackage{fancyvrb}
\usepackage{hyperref}
\usepackage[round]{natbib}
\usepackage{sectsty}

\hypersetup{
    bookmarks=true,         
    unicode=true,           
    colorlinks=true,        
    linkcolor=blue,         
    citecolor=blue,         
    filecolor=blue,         
    urlcolor=blue           
}

\usepackage{stackengine}
\stackMath
\newcommand\tenq[2][1]{%
\def\useanchorwidth{T}%
\ifnum#1>1%
\stackunder[0pt]{\tenq[\numexpr#1-1\relax]{#2}}{\!\scriptscriptstyle\thicksim}%
\else%
\stackunder[1pt]{#2}{\!\scriptstyle\thicksim}%
\fi%
}

\makeatletter
\DeclareRobustCommand\widecheck[1]{{\mathpalette\@widecheck{#1}}}
\def\@widecheck#1#2{%
    \setbox\z@\hbox{\m@th$#1#2$}%
    \setbox\tw@\hbox{\m@th$#1%
       \widehat{%
          \vrule\@width\z@\@height\ht\z@
          \vrule\@height\z@\@width\wd\z@}$}%
    \dp\tw@-\ht\z@
    \@tempdima\ht\z@ \advance\@tempdima2\ht\tw@ \divide\@tempdima\thr@@
    \setbox\tw@\hbox{%
       \raise\@tempdima\hbox{\scalebox{1}[-1]{\lower\@tempdima\box
\tw@}}}%
    {\ooalign{\box\tw@ \cr \box\z@}}}
\makeatother

\DeclareMathOperator*{\argmax}{arg\,max}

\def\given{\,|\,}

\def\Biggiven{\,\Big{|}\,}
\def\tr{\mathop{\text{tr}}\kern.2ex}

\def\P{{\mathrm P}}

\def\E{{\mathrm E}}
\def\R{{\mathbbm R}}

\def\d{{\mathrm d}}

\newcommand{\sgn}{\mathrm{sgn}}

\renewcommand{\Pr}{\mathrm{P}}

\newcommand\yestag{\addtocounter{equation}{1}\tag{\theequation}}
\newcolumntype{L}[1]{>{\raggedright\let\newline\\\arraybackslash\hspace{0pt}}m{#1}}
\newcolumntype{C}[1]{>{  \centering\let\newline\\\arraybackslash\hspace{0pt}}m{#1}}
\newcolumntype{R}[1]{>{ \raggedleft\let\newline\\\arraybackslash\hspace{0pt}}m{#1}}
\newcolumntype{d}[1]{D{.}{.}{#1}}
\newcolumntype{H}{>{\setbox0=\hbox\bgroup}c<{\egroup}@{}}
\newcolumntype{Z}{>{\setbox0=\hbox\bgroup}c<{\egroup}@{\hspace*{-\tabcolsep}}}
\newcolumntype{b}{X}
\newcolumntype{s}{>{\hsize=.5\hsize}X}

\numberwithin{equation}{section}
\theoremstyle{plain}
\newtheorem{theorem}{Theorem}[section]

\newtheorem{lemma}[theorem]{Lemma}
\newtheorem{proposition}[theorem]{Proposition}

\theoremstyle{definition}

\newtheorem{remark}[theorem]{Remark}

\graphicspath{{./fig3/}}

\allowdisplaybreaks

\begin{document}

\setlength{\abovedisplayskip}{10pt}
\setlength{\belowdisplayskip}{10pt}
\setlength{\abovedisplayshortskip}{10pt}
\setlength{\belowdisplayshortskip}{10pt}
\hypersetup{colorlinks,breaklinks,urlcolor=blue,linkcolor=blue}

\title{\LARGE On universal inference in Gaussian mixture models}

\author{
Hongjian Shi\thanks{TUM School of Computation, Information and Technology, Technical University of Munich; e-mail: \url{hongjian.shi@tum.de}}~~~and~
Mathias Drton\thanks{TUM School of Computation, Information and Technology and Munich Data Science Institute, Technical University of Munich; Munich Center for Machine Learning; e-mail: \url{mathias.drton@tum.de}}.
}

\date{}

\maketitle

\begin{abstract} 
A recent line of work provides new statistical tools based on game-theory and achieves safe anytime-valid inference without assuming regularity conditions. In particular, the framework of universal inference proposed by \citet{MR4242731} offers new solutions to testing problems by modifying the likelihood ratio test in a data-splitting scheme. 
In this paper, we study the performance of the resulting split likelihood ratio test under Gaussian mixture models, which are canonical examples for models in which classical regularity conditions fail to hold. We establish that under the null hypothesis, the split likelihood ratio statistic is asymptotically normal with increasing mean and variance. Contradicting the usual belief that the flexibility of universal inference comes at the price of a significant loss of power, we prove that universal inference surprisingly achieves the same detection rate $(n^{-1}\log\log n)^{1/2}$ as the classical likelihood ratio test.
\end{abstract}

{\bf Keywords:} $E$-value, mixture model, likelihood ratio test, singularity, universal inference.

\section{Introduction}\label{sec:intro}

Mixture models have a long history and have found applications in many fields. Numerous monographs are dedicated to the topic; e.g., \citet{MR0624267}, \citet{MR0838090}, \citet{MR0926484}, \citet{zbMATH05569717}, \citet{MR1417721,MR2392878}, \citet{MR1789474}, \citet{MR1684363}, \citet{MR2265601}, \citet{Schlattmann2009}, and \citet{MR4704098}, among others.  
The problem of estimating parameters in mixture models was already prominently studied by \citet{doi:10.1098/rsta.1894.0003} who proposed the method of moments.  In contrast, the history of testing the homogeneity of mixture models is shorter and usually traced back to \citet{MR0216727} and \citet{doi:10.1207/s15327906mbr0503_6,wolfe1971monte}; other pioneering works include \citet{binder1977cluster} and \citet{MR0443145-hartigan}. 

\citet[Chap.~4]{zbMATH05569717} and also \citet[Sec.~1]{MR3815469} discuss the different approaches that have been developed to test homogeneity of mixture models. Our focus in this paper will be on likelihood ratio tests (LRTs).  We begin with a review of key developments on LRTs (Section~\ref{subsec:intro:lrt}) and then mention recent new advances that give rise to an alternative approach termed universal inference (Sections~\ref{subsec:intro:univ}--\ref{subsec:intro:univ:mix}).  The main contributions of this paper will provide the first results on the power of universal inference in comparison to the traditional LRT for homogeneity.

\subsection{Likelihood ratio tests}
\label{subsec:intro:lrt}

Ever since \citet{zbMATH03029575} derived the asymptotic distribution of the likelihood ratio (LR) statistic $\lambda_n$ when testing composite hypotheses for regular models, LRTs have been applied ubiquitously and, in particular, in mixture models.
Emerging from the problem of clustering, a frequently used Gaussian mixture model is given by densities
\begin{equation}\label{eq:hetero_model}
f_{p,t_1,t_2,\sigma_1^2,\sigma_2^2}(x) = (1-p) \phi(x;t_1,\sigma_1^2)+p \phi(x;t_2,\sigma_2^2),
\end{equation}
where $p\in[0,1]$ is the mixture weight and $\phi(x;t,\sigma^2):=(2\pi\sigma^2)^{-1/2}\exp\{-(x-t)^2/(2\sigma^2)\}$ is the normal density for mean $t$ and variance $\sigma^2$.  However, the likelihood function of the (heteroscedastic) model~\eqref{eq:hetero_model} is unbounded for any given sample of size $n$, as noted by \citet{MR0086464}.  A likelihood ratio based on a global maximum of the likelihood function, thus, does not exist.  A homoscedastic Gaussian model given by
\begin{equation}\label{eq:homo_model}
f_{p,t_1,t_2,\sigma^2}(x) = (1-p) \phi(x;t_1,\sigma^2)+p \phi(x;t_2,\sigma^2),
\end{equation}
which has the advantage of possessing a bounded likelihood function, is accordingly sometimes preferred.  Often with further specializations, this homoscedastic model has also played an important role for theoretical studies on the behavior of the LRT in mixture models.

Let $X_1,\dots,X_n$ be an i.i.d.~sample comprised of $n$ real-valued random variables.  For Gaussian mixture models, testing homogeneity means testing whether $X_1,\dots,X_n$ are drawn from a single normal population versus a mixture. Expressed in terms of the parameters of the homoscedastic model from \eqref{eq:homo_model}, the null hypothesis is obtained by taking $p=0$ or $p=1$ or $t_1=t_2$. 

\paragraph*{Finite sample simulations.}
Already \citet{scott1971clustering} and \citet{wolfe1971monte} noticed that the regularity conditions of \citet{zbMATH03029575} no longer hold for the mixture model, which prompted numerical explorations. 
Based on simulation results, \citet{doi:10.1207/s15327906mbr0503_6, wolfe1971monte} abandoned Wilks' approximation for the distribution of the LR statistic. 
\citet{wolfe1971monte} and \citet{MR0443145-hartigan} focused on the homoscedastic Gaussian model~\eqref{eq:homo_model}, with extension to $m$-dimensional observations and mean vectors. 
The former author conjectured an asymptotic chi-squared distribution but with the degrees of freedom doubled to $2m$ (the formula excludes the mixing proportion when determining the degrees of freedom).
The latter author guessed the asymptotic distribution is between $\chi_{m}^2$ and $\chi_{m+1}^2$.

Through additional simulations, \citet{doi:10.1207/s15327906mbr1602_3} rejected Hartigan's 1977 conjecture and claimed Wolfe's approximation is valid only when the sample size is of the order of ten times the difference in the number of parameters.  \citet{mclachlan1987bootstrapping} argued Wolfe's approximation is well-suited for the homoscedastic case \eqref{eq:homo_model} but not for the heteroscedastic case \eqref{eq:hetero_model} and suggested using~$\chi_6^2$ instead for the case $m=1$.  
\citet{MR0981003} revisited the univariate homoscedastic case and reported that for $n<1000$ the distribution of $\lambda_n$ is close to $\chi_2^2$ rather than $\chi_1^2$. 

To increase the chance of reliable convergence, \citet{MR0790575} suggested, instead of compact parameter space assumption introduced by \citet{MR0600553}, imposing the restriction that $\min_{i,j}({\sigma_i}/{\sigma_j}) \ge c$ for some constant $c > 0$ and proved that a strongly consistent global maximizer exists for any suitable choice of $c$.
\citet{zbMATH00788192} noted that, using the restriction $\min(\sigma_1^2,\sigma_2^2) \ge c' > 0$, the simulated distribution of the LR statistic is between the chi-squared distributions $\chi_4^2$ and $\chi_5^2$ for $c'= 10^{-6}$ (Figure~1 therein) and between~$\chi_5^2$ and~$\chi_6^2$ for $c'=10^{-10}$ (Figure~2 therein).  

\paragraph*{Asymptotic theory.}
In very special submodels, it is not hard to derive the asymptotic distribution of the LR statistic.  For example, \citet[Sec.~7]{MR0800514} and \citet[Chap.~4]{zbMATH05569717} treat models where the mixing proportion $p$ is the only unknown parameter and derive the asymptotic distribution of the LR statistic $\lambda_n$ for testing whether $p=0$; compare also \citet{MR0065087}.
On the other hand, \citet{MR1209483} provide the asymptotic distribution for models in which the mixing proportion $p$ is known.
In both cases, the asymptotic distributions are mixtures of chi-squared distributions of different degrees of freedom.  
However, the asymptotic behavior of $\lambda_n$ when neither parameter is known was not explored until the work of \citet{MR0822065}.

\citet{MR0822065} not only considered testing homogeneity under a more general model
\begin{equation}\label{eq:general_model}
f_{p,\theta_1,\theta_2}(x) = (1-p) g(x;\theta_1)+p g(x;\theta_2),
\end{equation}
with not necessarily Gaussian component densities $g$,
but also gave the first correct asymptotic expression of the corresponding LR statistics. However, their work assumes that~$\theta_1$ and~$\theta_2$ are bounded and separated ($|\theta_1-\theta_2|\ge\epsilon>0$). 
\citet{MR1468112}, \citet{MR1704563}, \citet{MR1994731}, and \citet{MR1842105} explored how to remove the separation condition in the general case \eqref{eq:general_model}.  \citet{MR1841402}, \citet{GoussanouYansounouLazareFranck2001Eadt} and \citet{MR2010439} offered solutions in the Gaussian case.

Relaxing the assumption of compactness of the parameter space is more challenging. 
\citet{MR0822066}, in a work that coincidentally appeared in the same proceedings as that of \citet{MR0822065}, considered the contaminated Gaussian mixture model
\begin{equation}
\label{eq:intro:contam}
f_{p,t}(x) = (1-p) \phi(x;0,1)+p \phi(x;t,1);
\end{equation}
see also \citet{berman1986some} and \citet{MR1093694}. \citet{MR0822066} intrinsically proved that the quadratic approximation of the LR statistic tends to infinity in probability.  The quadratic approximation also emerges in related work of \citet{MR0501523, MR0885917}; see also \citet[Chap.~5.4]{MR0838090}.
\citet{MR0822066} also conjectured that (1) the quadratic approximation has exactly the order of $O(\log\log n)$, and~(2) the quadratic approximation and the LR statistic are stochastically equivalent.
\citet{10015228522} proved the first part of Hartigan's 1985 conjecture.  
\citet{MR1468112}, \citet[Chap.~3]{MR2695869}, \citet{MR1841402}, \citet{MR1842105}, and \citet{MR1955343}, among others, partially proved the second part of Hartigan's 1985 conjecture by imposing restrictions on parameters, and \citet{MR2058122} completed the whole story by proving the original conjecture.  
In addition, \citet{MR2146091} and \citet{MR2265342} studied the power of the LRT based on the asymptotic distribution derived by \citet{MR2058122} under two different types of local alternative hypotheses.

\paragraph*{Constrained, restricted and modified likelihood ratio test.}

In terminology adapted from \citet{MR2146091}, an LRT formed under a compactness constraint on the location parameter(s) of a mixture model is often referred to as a {\em constrained likelihood ratio test}.  Such tests were explored, to name a few, by \citet{MR1842105}, \citet{MR1841402}, \citet{MR1955343}, and \citet{MR2265342}.
On the other hand, \citet{MR1387544, MR1812716} and \citet{MR1366707} considered the {\em restricted likelihood ratio test}, investigating the asymptotic distribution theory of the LRT under the restriction that $\epsilon\le p\le 1-\epsilon$ for a fixed~$\epsilon > 0$.

\citet{MR1671968} and \citet{MR1811988} introduced the {\em modified (penalized) likelihood ratio test} by adding a penalty term in $p$.  This term penalizes mixing proportions close to~$0$~or~$1$, and when the mixing proportion is unidentifiable under a null hypothesis, its estimate is moved towards a unique minimizer of the penalty.
The modified LRT is discussed extensively in \citet{MR2126840} and \citet{MR2382561}.
Based on the penalized likelihood ratio, \citet{MR2543701} and \citet{MR2752604} developed another variant of an LRT, namely, the EM test for testing  (Gaussian) mixture models.

\subsection{Universal inference}
\label{subsec:intro:univ}

Taking a perspective of game-theoretic statistical inference and safe anytime-valid inference (SAVI), a series of recent papers by 
\citet{MR4255905}, 
\citet{MR4298879, MR4186488, MR4596762},
\citet{MR4825000}, \citet{MR4590499}, \citet{MR4364896}, 
\citet{MR4242731}, \citet{ramdas2020admissible, MR4364897, MR4665027}, \citet{MR4660691}, 
and \citet{MR4460577}
proposed and examined the notion of an $e$-value.  An $e$-value is a nonnegative random variable with expectation no larger than $1$ under (any) null hypothesis.
By Markov's inequality, comparing an $e$-value with the threshold $1/\alpha$ yields a level $\alpha$ test of the considered null hypothesis.  

\citet{MR4242731} contributed to this field by proposing the construction of a split likelihood ratio that constitutes an $e$-value.  This leads to a split likelihood ratio test (SLRT) that is finite-sample-valid under virtually no regularity conditions.  Fittingly, the inferential methodology is termed universal inference.

A natural question is whether the universality of the SLRT comes at a price of a severe loss of power.
Indeed, \citet{MR4242731} already commented, ``our methods may not be optimal, though we do not yet fully understand how close to optimal they are beyond special cases (uniform, Gaussian).''
\citet{MR4527023} and \citet{MR4529724} provide empirical evidence to illustrate that the split likelihood ratio test (SLRT) may be highly conservative.
\citet{MR4588356}, on the other hand, studied the performance of universal inference under regular settings, where the classical likelihood ratio test itself applies, and concluded the power is reasonable by showing the ratio of squared radii of confidence sets of SLRT and LRT is bounded in mean.

\subsection{Universal inference in mixture models}
\label{subsec:intro:univ:mix}

\citet{MR4242731} described numerous settings, 
including testing the number of components in mixture models, 
in which the universal LRT is the first hypothesis test with finite sample validity.  In fact, for testing the number of mixture components even asymptotically valid competitors are difficult to construct.
 
\citet{MR4242731} and \citet{MR4529724} performed simulation studies and gave some theoretical justifications for Gaussian mixture models.  However, both assume that the mixing proportion $p$ is known to be $1/2$.  In that case, the mixture model is nearly regular.
\citet{MR4588356} claimed their study 
in the regular settings ``as a precursor to studying the power in these important, but currently intractable, settings''.  The objective of this paper is to investigate the performance of universal inference in non-regular Gaussian mixture models.

\subsection{Outline of the paper}

The rest of the paper is organized as follows. Section~\ref{sec:2} reviews universal inference and the split likelihood ratio test proposed by \citet{MR4242731}.  Section~\ref{sec:3} turns to the problem of testing homogeneity in Gaussian mixtures.  Specifically, we take up the classical model from~\eqref{eq:intro:contam} and test the null hypothesis of a single standard normal population. We rigorously develop the large-sample theory of the SLRT for this problem and obtain an asymptotic normality result under the null (with diverging mean and variance).  In Sections~\ref{sec:4} and~\ref{sec:5}, we conduct local power analyses for the SLRT for two cases of alternatives.  A key finding emerges from the non-contiguous alternatives treated in Section~\ref{sec:5}, which offer a surprising result on the power of the SLRT:  It is able to achieve the same detection boundary as the classical LRT. 
Our numerical results in Section~\ref{sec:6} illustrate this fine-grained analysis for large sample sizes and additionally consider variations of the testing problem are provided in Section~\ref{sec:6}.  
We conclude in Section~\ref{sec:discussion} with a discussion of our results, which we consider intriguing positive results on the potential of universal inference to rigorously solve challenging testing problems in mixture models.

\paragraph*{Notations.}
A list of all commonly used notations in the paper is given in Table~\ref{tab:1}.  Detailed descriptions can be found in the later sections. 

{
\renewcommand{\tabcolsep}{1pt}
\renewcommand{\arraystretch}{1.20}
\begin{table}[t] 
\centering
\caption{Commonly used notations}\label{tab:1}
\begin{tabular}{L{0.6875in}L{2.125in}L{0.6875in}L{2.25in}}
Notation                         & Definition                                                 & 
Notation                         & Definition \\
$Z_{i,k}(t)$                     & $\exp(tX_{i,k}-t^2/2)-1$                                   & 
$W_{i,k}^*(t)$                   & $(e^{t^2}-1)^{-1/2}Z_{i,k}(t)$~~~for~$t\ne0$ \\ 
$W_{i,k}(t)$                     & $Z_{i,k}(t)\exp(-t^2/2)$                                   & 
$W_{i,k}^*(0)$                   & $X_{i,k}$ \\
$L_{n,k}(\eta,t)$                & $\sum_{i=1}^{n_k}\log\big\{1+\eta W_{i,k}(t)\big\}$        & 
$L_{n,k}^*(\eta, t)$             & $\sum_{i=1}^{n_k}\log\big\{1+\eta W_{i,k}^*(t)\big\}$ \\
$\big(\widehat \eta_{n,k},\widehat t_{n,k}\big)$
                                 & $\argmax_{\eta\in[0,\exp(t^2/2)], t\in\R}L_{n,k}(\eta, t)$ & 
$\big(\widehat \eta^*_{n,k},\widehat t^*_{n,k}\big)$ 
                                 & $\argmax_{\eta\in[0,(e^{t^2}-1)^{1/2}], t\in\R}L_{n,k}^*(\eta, t)$ \\
$\widehat \eta_{n,k}(t)$         & $\argmax_{\eta\in[0,\exp(t^2/2)]}L_{n,k}(\eta, t)$         & 
$\widehat\eta^*_{n,1}(t)$        & $\argmax_{\eta\in[0,(e^{t^2}-1)^{1/2}]}L_{n,k}^*(\eta, t)$ \\
$S_{n,k}(t)$                     & $n_k^{-1/2}\sum_{i=1}^{n_k}W_{i,k}(t)$                     & 
$S_{n,k}^*(t)$                   & $n_k^{-1/2}\sum_{i=1}^{n_k}W_{i,k}^*(t)$ \\
$V_{n,k}(t)$                     & $n_k^{-1}  \sum_{i=1}^{n_k}W_{i,k}(t)^2$                   & 
$V_{n,k}^*(t)$                   & $n_k^{-1}  \sum_{i=1}^{n_k}W_{i,k}^*(t)^2$ \\[-1.75mm]
$M_{n,k}$                        & $\sup_{t\in\R} S_{n,k}(t)$ & \smash[b]{\raisebox{2mm}{\rule{2.75in}{0.5pt}}}\\[-1.75mm]
$\lambda_{n,k}$                  & 2\,$\sup_{\eta\in[0,\exp(t^2/2)], t\in\R} L_{n,k}(\eta,t)$ &
$\lambda_{n,k}(I)$               & 2\,$\sup_{\eta\in[0,\exp(t^2/2)],|t|\in I} L_{n,k}(\eta,t)$ \\
$x_{+}$                          & $\max\{x,0\}$                                              &
$t_0$                            & a constant greater than $4$ \\
$\log_{(2)} n$                   & $\log \log n$                                              &
$\log_{(3)} n$                   & $\log \log \log n$ \\
$\epsilon_{1,n}$                 & $(\log n)^{-1}$                                            &
$\epsilon_{2,n}$                 & $(\log \log n)^{-1}$ \\
$I_{1,n}$                        & $[0, t_0]$                                                 &
$I_{2,n}$                        & $[t_0, \sqrt{(\log n)/2}\,]$ \\
$I_{3,n}$                        & $[\sqrt{(\log n)/2}, \sqrt{2\log n}\,]$                    &
$I_{4,n}$                        & $[\sqrt{2\log n}, +\infty)$ \\
$c_{1,n}$                        & $2\sqrt{\smash[b]{\log_{(3)} n}}$                                     & 
$c_{2,n}$                        & $\sqrt{(\log n)/2}-2\sqrt{\smash[b]{\log_{(2)} n}}$ \\
$A_{1,n}$                        & $[t_0, c_{1,n}]$                                   & 
$A_{2,n}$                        & $[c_{1,n}, c_{2,n}]$ \\
$A_{2,n}^{\sqsupset}(\ell)$      & $[c_{1,n},\, (1-\ell)c_{1,n}+\ell c_{2,n}]$        &
$A_{2,n}^{\sqsubset}(\ell)$      & $[(1-\ell)c_{1,n}+\ell c_{2,n},\, c_{2,n}]$ \\
$A_{3,n}$                        & $[c_{2,n}, \sqrt{(\log n)/2}]$                     &
$D_n$                            & $[0,c_{1,n}]\cup [c_{2,n}, \infty)$ \\[-2mm]
\end{tabular}
\end{table}
}

\section{Background on the split likelihood ratio test}\label{sec:2}

Let $\{\P_{\theta}: \theta \in \Theta\}$ be a parametric statistical model, with parameter space $\Theta \subseteq \R^d$.
The distributions~$\P_{\theta}$ are assumed to have probability densities $f_{\theta}$ with respect to a common dominating measure $\nu$.  
Assume the observations $X_1, \dots, X_n$ are independent and identically distributed (i.i.d.) according to an unknown distribution $\P_{\theta}$ in the model, and suppose that, given a subset~$\Theta_0\subsetneq\Theta$, we are interested in the testing problem 
\begin{equation}\label{eq:testing}
H_0: \theta \in \Theta_0 \quad\text{versus}\quad H_1:\theta\in\Theta\setminus\Theta_0.
\end{equation}

Let $\ell(\theta) = \sum_{i=1}^{n} \log f_{\theta}(X_{i})$ be the log-likelihood function.  The (classical) LR statistic for~\eqref{eq:testing} is given by
\[
\lambda_n:=2\Big\{\sup_{\theta\in\Theta}\ell({\theta})-\sup_{\theta\in\Theta_0}\ell({\theta})\Big\}.
\]
For regular problems, asymptotically valid LRTs may be constructed via Wilks' theorem, i.e., the fact that the distribution of~$\lambda_n$  converges to the chi-squared distribution~$\chi_m^2$ under the null hypothesis.  However, when regularity conditions fail, it can be difficult to provide theoretical insights on the distribution of likelihood ratios, and standard bootstrapping is not necessarily valid; see, e.g., \citet{MR2860333}.
These issues are particularly pressing for mixture models.

Universal inference proposed by \citet{MR4242731} circumvents these inferential problems by modifying the likelihood ratio test in a data-splitting scheme. 
The data are divided into two parts, $D_0$ for inference and $D_1$ for estimation. 
For this split, choose a fraction $m_0 \in (0,1)$ and partition the data into two disjoint subsets 
$D_0 = \{ X_{1,0},\ldots,X_{\lfloor m_0n \rfloor,0} \}$ and 
$D_1 = \{ X_{1,1},\ldots,X_{\lceil m_1n \rceil,1} \}$, where $m_1:= 1-m_0$.  
We will write $n_0$ for $\lfloor m_0n \rfloor$ and $n_1$ for $\lceil m_1n \rceil$ to shorten notation.  
Let
\[
\ell_k(\theta) = \sum_{i =1}^{n_k} \log f_{\theta}(X_{i,k}), \quad k=0,1,
\]
be the likelihood functions based on $D_0$ and $D_1$, respectively.  Let $\widehat{\theta}_{n,1}:= \argmax_{\theta \in \Theta} \ell_{1}(\theta)$ be the maximum likelihood estimator (MLE) of $\theta$ under the full model and based on $D_1$, and let $\widehat{\theta}_{n,0} := \argmax_{\theta \in \Theta_0} \ell_{0}(\theta)$ be the MLE of $\theta$ under $H_0$ and based on $D_0$.  Now the {\em split likelihood ratio statistic} (SLR statistic) is defined as
\begin{equation}\label{eq:SLRT}
\lambda_n^{\rm split}:=2\Big\{\ell_{0}(\widehat{\theta}_{n,1})-\ell_{0}(\widehat{\theta}_{n,0})\Big\}.
\end{equation}

As shown in \citet{MR4242731} and \citet{MR4527023}, under the null hypothesis $H_0: \theta \in \Theta_0$, it holds for any positive integer $n$ that
\begin{equation}\label{eq:jiaevalue}
\E_{\theta}[\exp(\lambda_n^{\rm split}/2)]
   = \E_{\theta}\Bigg[\frac{\prod_{i=1}^{n_0} f_{\widehat{\theta}_{n,1}}(X_{i,0})}{ \prod_{i=1}^{n_0} f_{\widehat{\theta}_{n,0}}(X_{i,0})}\Bigg]
   \leq \E_{\theta}\Bigg[\E_{\theta}\Bigg[\frac{\prod_{i=1}^{n_0} f_{\widehat{\theta}_{n,1}}(X_{i,0})}{ \prod_{i=1}^{n_0} f_{\theta}(X_{i,0})}\Bigg|D_1\Bigg]\Bigg]
   \leq 1,
\end{equation}
where we use the fact that $\widehat{\theta}_{n,1}$ is fixed conditioning on $D_1$, and for any fixed $\theta^* \in \Theta$ it holds that 
\[
   \E_{\theta}\Bigg[\frac{\prod_{i=1}^{n_0} f_{\theta^*}(X_{i,0})}{\prod_{i=1}^{n_0} f_{\theta}(X_{i,0})}\Bigg] \leq \int \prod_{i=1}^{n_0} f_{\theta^*}(x_{i,0}) {\rm d}x_{1,0}\cdots {\rm d}x_{n_0,0}
   = 1.
\]
In other words, the likelihood ratio corresponding to $\lambda_n^{\rm split}$ is an $e$-value; recall Section~\ref{subsec:intro:univ}.
An application of Markov's inequality yields for any $\alpha \in (0, 1)$ and any positive integer $n$,
\begin{equation}\label{eq:univ_critc}
\P_{\theta}(\lambda_n^{\rm split}> -2\log\alpha)\le \alpha.
\end{equation}
Accordingly, the test given by $\mathds{1} (\lambda_n^{\rm split}> -2\log\alpha)$ is finite-sample--valid at significance level~$\alpha$.  The test is universal in the sense that the critical point $-2\log\alpha$ does not depend on the statistical model and the true parameter.  However, it is not obvious whether the test can achieve competitive power.  This issue motivates our subsequent study of the SLRT under Gaussian mixture models.

\section{Asymptotic null distribution}\label{sec:3}

In the sequel, we consider the contaminated Gaussian mixture model
\begin{equation}\label{eq:two_contamin}
f_{p,t}(x) = (1-p) \phi(x;0,1)+p \phi(x;t,1)
\end{equation}
where the mixture weight $p\in[0,1]$ and the mean $t\in\R$ are unknown parameters.  Given an i.i.d.~sample $X_1,\dots,X_n$, we consider the homogeneity testing problem 
\begin{equation}
\label{eq:homogeneity:testing}
 H_0: p=0~\text{or}~t=0~~~\text{against}~~~
 H_1: p\in(0,1), t\in\R\backslash\{0\}.
\end{equation}
We emphasize that $H_0$ specifies a standard normal distribution.
In the main result of this section, we derive the asymptotic null distribution of the split likelihood ratio statistic for the homogeneity problem in~\eqref{eq:homogeneity:testing}.  Whereas an extreme value distribution has been found for the ordinary LRT, our  Theorem~\ref{thm:null} gives a normal limit for the SLRT.

We begin by stating seminal results on the standard LRT.  Suppose that the null hypothesis $H_0$ is true; in other words, $X_1, \dots, X_n$ are i.i.d.~standard normal random variables. 
\citet{MR0822066} correctly conjectured that the LR statistic diverges to $+\infty$ in probability at the order of $O(\log\log n)$, instead of converging to some chi-squared distribution.
In addition,
the following important results hold for the {\em likelihood ratio statistic}~$\lambda_n$:

{{{
\begin{proposition}[Main theorem in \citet{10015228522}]\label{prop:theoremBC}
Denote 
\[
S_n(t):=n^{-1/2}\sum_{i=1}^{n}\big(e^{tX_i-t^2/2}-1\big)e^{-t^2/2}~~~\text{and}~~~
M_n:=\sup_{t\in\R} S_n(t).
\]
Then
\[\yestag\label{eq:theoremBC}
\lim_{n\to\infty}\P_{H_0}\Big\{\sqrt{\log_{(2)}n}\Big(M_n-\sqrt{\log_{(2)}n}\,\Big)+\log(\sqrt{2}\pi)\leq x\Big\}=\exp\{-\exp(-x)\}, \quad x\in\R.
\]
Moreover, 
\[\yestag\label{eq:theoremBCmore}
\sup_{|t|\in D_n}S_n(t)=o_\P\Big(\sqrt{\log_{(2)}n}\,\Big),
\]
where $D_n=\Big[0,2\sqrt{\log_{(3)}n}\,\Big]\cup \Big[\sqrt{(\log n)/2}-2\sqrt{\log_{(2)} n}, \infty\Big)$.
\end{proposition}

\begin{proposition}[Theorem~2 in \citet{MR2058122}]\label{prop:theoremLS}
The likelihood ratio statistic $\lambda_n$ for testing homogeneity in the contaminated
Gaussian mixture model \eqref{eq:two_contamin} is given by
\[
2\sup_{p\in[0,1],t\in\R} \sum_{i=1}^{n}\log\big\{1+p\big(e^{tX_i-t^2/2}-1\big)\big\}
\]
and satisfies
\[\yestag\label{eq:theoremLS}
\lim_{n\to\infty}\P_{H_0}\{\lambda_n - \log \log n + \log(2\pi^2) \leq x\} = \exp\{-\exp(-x/2)\}, \quad x\in\R.
\]
Consequently,
\[
\lim_{n\to\infty}\P_{H_0}\{\lambda_n > c_{n,\alpha}\} = \alpha,
\]
where the critical value is defined as 
\begin{equation}\label{eq:def_c}
c_{n,\alpha} =  \log \log n - \log(2\pi^2) - 2\log\log(1-\alpha)^{-1}.
\end{equation}
\end{proposition}

The main idea of proving Proposition~\ref{prop:theoremBC} is to relate $S_n$ to a Gaussian process $\widetilde S$ and to show $S_n(t)$ behaves like $\widetilde S(t)$ for $|t|\le \sqrt{(\log n)/2}$ and is small otherwise.  
The proof of Proposition~\ref{prop:theoremLS}, given Proposition~\ref{prop:theoremBC}, is concluded by justifying the asymptotic equivalence between $\lambda_n$ and $M_n^2$; the latter is actually the quadratic approximation of $\lambda_n$.

\smallskip

We are now ready to explore the distribution of the SLR statistic~$\lambda_n^{\rm split}$, which is defined as in \eqref{eq:SLRT}.
For our problem, the SLR statistic $\lambda_n^{\rm split}$ can be rewritten as
\(
\lambda_n^{\rm split}=2L_{n,0}^{\rm orig}(\widehat p_{n,1}, \widehat t_{n,1}),
\)
where
\begin{align*}
L_{n,k}^{\rm orig}(p, t)&:=\sum_{i=1}^{n_k}\log\Big\{1+p Z_{i,k}(t)\Big\},~~~
\Big(\widehat p_{n,1},\widehat t_{n,1}\Big):=\argmax_{p\in[0,1],\,t\in\R}L_{n,1}^{\rm orig}(p, t),\\
\text{and}~~~Z_{i,k}(t)&:=\exp(tX_{i,k}-t^2/2)-1.
\end{align*}
Furthermore, we introduce a re-parameterization for the likelihood ratio $L_{n,k}^{\rm orig}(p,t)$: 
\begin{equation}\label{eq:repara1}
\eta = p \exp(t^2/2)
~~~\text{and}~~~
W_{i,k}(t)=Z_{i,k}(t)\exp(-t^2/2).
\end{equation}
Accordingly, the SLR statistic can be represented as
\(
\lambda_n^{\rm split}=2L_{n,0}(\widehat\eta_{n,1}, \widehat t_{n,1}),
\)
where
\begin{align*}
L_{n,k}(\eta,t)=\sum_{i=1}^{n_k}\log\Big\{1+\eta W_{i,k}(t)\Big\}
~~~\text{and}~~~
\Big(\widehat\eta_{n,1},\widehat t_{n,1}\Big):=\argmax_{\eta\in[0,\exp(t^2/2)],\,t\in\R}L_{n,1}(\eta, t).
\end{align*}
Throughout Sections~\ref{sec:3}--\ref{sec:4}, we will focus on the re-parameterization~\eqref{eq:repara1}.
In Section~\ref{sec:5}, we will use a slightly different re-parameterization, denoted by the superscript$\,^*$.
Let $\widehat\eta_{n,k}(t)$ be the maximum likelihood estimator of $\eta$ for a given $t$, i.e., 
\[
L_{n,k}(\widehat\eta_{n,k}(t), t)=\sup_{\eta\in [0,\exp(t^2/2)]} L_{n,k}(\eta,t),
\]
and define, for any subinterval $I\subseteq [0,\infty)$, 
\[
\lambda_{n,k}(I):= 2\sup_{\eta\in[0,\exp(t^2/2)],|t|\in I} L_{n,k}(\eta,t).
\]
}}}

The following result is the main result in this section.  It shows that the SLR statistic tends to $-\infty$ at the order of $O(\log\log n)$ and is asymptotically normal with increasing mean and variance.

\begin{theorem}\label{thm:null}
Suppose that $X_1, \dots, X_n$ are i.i.d.~standard normal random variables. 
The asymptotic null distribution of the SLR statistic is obtained as
\begin{equation}\label{eq:limitnull}
\frac{\lambda_n^{\rm split} + \frac{m_0}{m_1}\log \log n}{2\sqrt{\frac{m_0}{m_1}\log \log n}}
\stackrel{\sf d}{\longrightarrow}
N(0,1).
\end{equation}
\end{theorem}

\begin{proof}[Proof of Theorem~\ref{thm:null}]
From the notations in Table~\ref{tab:1}, recall that 
\[
A_{2,n}:=\Big[2\sqrt{\smash[b]\log_{(3)} n}, \sqrt{(\log n)/2}-2\sqrt{\smash[b]\log_{(2)} n}\,\Big].
\]
Now, uniformly over $|t|\in A_{2,n_1}$, the split likelihood function can be quadratically approximated:
\begin{align*}\yestag\label{eq:haolei}
2L_{n,0}(\widehat\eta_{n,1}(t),t)
=\;&2\sum_{i=1}^{n_0}\log\Big\{1+\widehat\eta_{n,1}(t) W_{i,0}(t)\Big\} \\
(\text{by Lemma~\ref{lem:LS2}(\ref{lem:LS2a})})~~~
=\;&2\widehat \eta_{n,1}(t)\sum_{i=1}^{n_0}W_{i,0}(t)
 -\big\{1+O_{\P}(\epsilon_{1,n})\big\}\widehat\eta_{n,1}(t)^2\sum_{i=1}^{n_0}W_{i,0}(t)^2 \\
=\;&2\sqrt{n_0}\widehat\eta_{n,1}(t) S_{n,0}(t)
 -\big\{1+O_{\P}(\epsilon_{1,n})\big\}n_0\widehat\eta_{n,1}(t)^2 V_{n,0}(t) \\
(\text{by Lemma~\ref{lem:LS2}(\ref{lem:LS2b})})~~~
=\;&2\sqrt{n_0}\widehat\eta_{n,1}(t) S_{n,0}(t)
 -\big\{1+o_{\P}(\epsilon_{2,n})\big\}n_0\widehat\eta_{n,1}(t)^2 \\
=\;&\sqrt{n_0}\widehat\eta_{n,1}(t) \Big[2S_{n,0}(t) 
 -\big\{1+o_{\P}(\epsilon_{2,n})\big\}\sqrt{n_0}\widehat\eta_{n,1}(t)\Big]\\
(\text{by Lemma~\ref{lem:LS2}(\ref{lem:LS2c})})~~~
=\;&\big\{1+o_{\P}(\epsilon_{2,n})\big\}\sqrt{\text{\ss}}S_{n,1}(t)_{+} \Big[2S_{n,0}(t) 
 -\big\{1+o_{\P}(\epsilon_{2,n})\big\}\sqrt{\text{\ss}}S_{n,1}(t)_{+}\Big],
\end{align*}
where $\text{\ss}:={m_0}/{m_1}$.
As $|\widehat t_{n,1}|\in A_{2,n_1}$ with probability tending to 1 (Lemma~\ref{lem:LS3}(\ref{lem:LS3a})), plugging $\widehat t_{n,1}$ into \eqref{eq:haolei} yields
\begin{align*}\yestag\label{eq:split-asymp}
\lambda_n^{\rm split}
=\;&2L_{n,0}(\widehat\eta_{n,1}(\widehat t_{n,1}),\widehat t_{n,1})\\
(\text{by Lemma~\ref{lem:LS3}(\ref{lem:LS3a})})~~~
=\;&\big\{1+o_{\P}(\epsilon_{2,n})\big\}\sqrt{\text{\ss}}S_{n,1}(\widehat t_{n,1})_{+} \Big[2S_{n,0}(\widehat t_{n,1}) 
-\big\{1+o_{\P}(\epsilon_{2,n})\big\}\sqrt{\text{\ss}}S_{n,1}(\widehat t_{n,1})_{+}\Big].
\end{align*}
Lemma~\ref{lem:LS3}(\ref{lem:LS3b}) gives, with probability tending to 1,
\begin{align*}\yestag\label{eq:erci}
S_{n,1}(\widehat t_{n,1})_{+}=S_{n,1}(\widehat t_{n,1})
               &=M_{n,1}+o_{\P}(\epsilon_{2,n_1}^{1/2})\\
(\text{by Proposition~\ref{prop:theoremBC}})~~~
               &=\sqrt{\log_{(2)}n_1}+O_{\P}(\epsilon_{2,n_1}^{1/2}).
\end{align*}
Furthermore, combining Equations~\eqref{eq:split-asymp} and~\eqref{eq:erci} and the fact
$S_{n,0}(\widehat t_{n,1})\stackrel{\sf d}{\longrightarrow} N(0,1)$ (Lemma~\ref{lem:LS1}), we deduce
\begin{align*}\yestag\label{eq:fenjie1}
\lambda_n^{\rm split}
=\;&2\sqrt{\text{\ss}}S_{n,1}(\widehat t_{n,1})_{+} S_{n,0}(\widehat t_{n,1}) 
 -\Big\{\sqrt{\text{\ss}}S_{n,1}(\widehat t_{n,1})_{+}\Big\}^2+o_{\P}(1).
\end{align*}
Therefore, noticing the fact 
$
S_{n,1}(\widehat t_{n,1})_{+}\stackrel{\sf p}{\longrightarrow}\infty
$
(Equation~\eqref{eq:erci}), we have
\[
\frac{\lambda_n^{\rm split} + \Big\{\sqrt{\text{\ss}}S_{n,1}(\widehat t_{n,1})_{+}\Big\}^2}{2\sqrt{\text{\ss}}S_{n,1}(\widehat t_{n,1})_{+}}=S_{n,0}(\widehat t_{n,1})+o_{\P}(1)\stackrel{\sf d}{\longrightarrow}
N(0,1),
\]
where the last step is by Lemma~\ref{lem:LS1}. Using Equation~\eqref{eq:erci} once again, we deduce
\begin{align*}
\{S_{n,1}(\widehat t_{n,1})_{+}\}^2=\log_{(2)}n_1+O_{\P}(1),
\end{align*}
and thus
\[
\frac{\lambda_n^{\rm split} + \text{\ss}\log \log n_1}{2\sqrt{\text{\ss}\log \log n_1}}
\stackrel{\sf d}{\longrightarrow}
N(0,1).
\]
Simple calculation yields the desired result \eqref{eq:limitnull}.
\end{proof}

\begin{remark}
As a direct corollary of Theorem~\ref{thm:null}, if we adopt the asymptotic critical point from the asymptotic null distribution \eqref{eq:limitnull}, namely, 
\begin{equation}\label{eq:def_csplit}
c_{n,\alpha}^{\rm split}:= 2\sqrt{\frac{m_0}{m_1}\log \log n}\times
\Phi^{-1}({1-\alpha})- \frac{m_0}{m_1}\log \log n,
\end{equation}
where $\Phi^{-1}(\cdot)$ denotes the quantile function of the standard normal distribution,
then the SLRT will have the asymptotic size of $\alpha$:
\[
\lim_{n\to\infty}\P_{H_0}\{\lambda_n^{\rm split} > c_{n,\alpha}^{\rm split}\}=\alpha.
\]
\end{remark}

\begin{remark}
Roughly speaking, 
\[
\frac{1}{2}\lambda_n^{\rm split}\approx_{\sf d} N\Big(-\frac{m_0}{2m_1}\log \log n,\frac{m_0}{m_1}\log \log n\Big),
\]
where a random variable following the right-hand side distribution has an exponential moment of exactly~$1$. This partially explains our later observation that the Markov inequality used in universal inference of \citet{MR4242731} (``poor man's Chernoff bound'') is not extremely conservative for the homogeneity problem. 
\end{remark}

The main reason why $\lambda_n^{\rm split}$ tends to $-\infty$ with rate $\log\log n$ (compare that $\lambda_n$ tends to $+\infty$ with rate $\log\log n$) is the next lemma, which shows that $S_{n,0}(\widehat t_{n,1})$ is asymptotically standard normal rather than tends to infinity in probability as $S_{n,1}(\widehat t_{n,1})$.

\begin{lemma}\label{lem:LS1}
$S_{n,0}(\widehat t_{n,1})\stackrel{\sf d}{\longrightarrow} N(0,1)$.
\end{lemma}

\begin{proof}[Proof of Lemma~\ref{lem:LS1}]
Notice that
\begin{align*}
S_{n,0}(t) 
&= n_0^{-1/2}\sum_{i=1}^{n_0}(e^{tX_i-t^2} -e^{-t^2/2}) \\
&= n_0^{1/2}\Big[\int_{-\infty}^{\infty} e^{tx-t^2}\d F_{n,0}(x) - \int_{-\infty}^{\infty} e^{tx-t^2}\d \Phi(x)\Big] \\
&= \int_{-\infty}^{\infty} e^{tx-t^2} \d B_{n,0}(\Phi(x)),
\end{align*}
with $\Phi(\cdot)$ the cumulative distribution function of the standard normal distribution,
\[
F_{n,0}(x):=n_0^{-1}\sum_{i=1}^{n_0}\mathds{1}\{X_{i,0}\le x\}~~~\text{and}~~~
B_{n,0}(u):=n_0^{1/2}\Big[n_0^{-1}\sum_{i=1}^{n_0}\mathds{1}\{\Phi(X_{i,0})\le u\} - u\Big].
\]
By the Hungarian construction \citep{MR0375412, MR0402883}, there exists a Brownian Bridge $B_0$ such that
\begin{equation}\label{eq:Hung}
\sup_{0\le u\le 1}|B_{n,0}(u)-B_0(u)|=O_{\P}(n_1^{-1/2}\log n_1)
\end{equation}
on a suitable probability space.\footnote{With slight abuse of notation, we will still use $\P$ to denote the extension of $\P$ by the Hungarian construction.}
Define $S_0(t)$ to be the process
\begin{equation}\label{eq:defS_0}
S_0(t):=\int_{-\infty}^{\infty} e^{tx-t^2}\d B_0(\Phi(x)).
\end{equation}
Moreover, define $\widetilde S_0(t)$ as $S_0(t)+\widetilde X e^{-t^2/2}$ with $\widetilde X$ standard normal and independent of the system.
Then $S_0(t)$ and $\widetilde S_0(t)$ are Gaussian processes with covariance functions
\[
\rho(s,t)=\exp\Big\{-\frac{(s-t)^2}{2}\Big\}-\exp\Big\{-\frac{s^2}{2}-\frac{t^2}{2}\Big\}
\quad\text{and}\quad
\widetilde\rho(s,t)=\exp\Big\{-\frac{(s-t)^2}{2}\Big\},
\]
respectively. Intuitively, by Equation~\eqref{eq:Hung}, $S_{n,0}(t)$, $S_0(t)$, and $\widetilde S_0(t)$ differ small for all appropriate $t$.
Furthermore, following Equation~(39) of \citet{10015228522}, we have, uniformly for~$t\in A_{2,n_0}$,
\begin{equation}\label{eq:jianjie}
S_{n,0}(t) - \widetilde S_0(t) = o_{\P}(\epsilon_{2,n_0}^{1/2}).
\end{equation}
Since $\widehat t_{n,1}$ is a random variable independent of $S_0$, we have
\[
\P(\widetilde S_0(\widehat t_{n,1})\le y)
=\int \P(\widetilde S_0(t)\le y)\d\P_{\widehat t_{n,1}}(t)
=\int \Phi(y)\d\P_{\widehat t_{n,1}}(t)
=\Phi(y).
\]
Then combining Equation~\eqref{eq:jianjie} and the fact that $|\widehat t_{n,1}|\in A_{2,n_0}$ with probability tending to 1 (Lemma~\ref{lem:LS3}(\ref{lem:LS3e})) concludes the proof. 
\end{proof}

The following lemma was used to prove Theorem~\ref{thm:null}.  It forms the basis for why the split likelihood ratio statistic is stochastically equivalent to its quadratic approximation.

\begin{lemma}\label{lem:LS2}
The following statements hold:
\begin{enumerate}[(i)]
\item\label{lem:LS2a}
$\sup_{|t|\in A_{2,n_1}}\max_{1\le i\le n_0} |\widehat\eta_{n,1}(t)W_{i,0}(t)| = O_{\P}\{(\log n_0)^{-1}\}$;
\item\label{lem:LS2b}
$V_{n,0}(t):=n_0^{-1}\sum_{i=1}^{n_0}W_{i,0}(t)^2=1+o_{\P}\{(\log\log n_0)^{-1}\}$, when $|t|\in A_{2,n_1}$; 
\item \label{lem:LS2c} $\sqrt{n_1}\widehat \eta_{n,1}(t) = \big\{1+o_{\P}(\epsilon_{2,n_1})\big\} S_{n,1}(t)_{+}$ uniformly over $|t|\in A_{2,n_1}$.

\end{enumerate}
\end{lemma}

\begin{proof}[Proof of Lemma~\ref{lem:LS2}]

Due to the symmetry, it suffices to prove all the results for $t\in A_{2,n_1}=
[2\sqrt{\smash[b]{\log_{(3)} n_1}}, \sqrt{(\log n_1)/2}-2\sqrt{\smash[b]{\log_{(2)} n_1}}]$.
\begin{enumerate}[(i)]
\item Lemma~2 in \citet{MR2058122} proves that
\begin{equation}\label{eq:renzhen}
\sup_{|t|\in A_{2,n_1}}\widehat\eta_{n,1}(t)=O_{\P}\{n_1^{-1/2}(\log n_1)^2\}.
\end{equation}
In addition, we have
\begin{align*}
\sup_{|t|\in A_{2,n_1}}\max_{1\le i\le n_0}|W_{i,0}(t)|
&\le \sup_{t\in A_{2,n_1}} \exp(t X_{(n_0),0}-t^2)+1\\
&=   \sup_{t\in A_{2,n_1}} \exp\{(X_{(n_0),0}/2)^2-(X_{(n_0),0}/2-t)^2\}+1,
\end{align*}
where $X_{(1),k},\dots,X_{(n_k),k}$ are the ascending order statistics of $X_{1,k},\dots,X_{n_k,k}$, $k=0,1$, and by Theorem~1.5.3 in \citet{MR691492},
\[
\lim_{C\to+\infty}\liminf_{n\to\infty}\P\Big(X_{(n_0),0}\in
\Big[\sqrt{2\log n_0}-\frac{\log\log n_0+C}{2\sqrt{2\log n_0}},
     \sqrt{2\log n_0}-\frac{\log\log n_0-C}{2\sqrt{2\log n_0}}\Big]\Big)=1.
\]
Accordingly, with probability tending to 1,
\begin{align*}
X_{(n_0),0}/2&\le \sqrt{(\log n_0)/2},\\
~~~\text{and}~~~
X_{(n_0),0}/2-t&\ge \sqrt{(\log n_0)/2}-\frac12-\Big(\sqrt{(\log n_1)/2}-2\sqrt{\log_{(2)} n_1}\,\Big)\\
               &\ge 2\sqrt{\log_{(2)} n_1} -1~~~\text{for}~t\in A_{2,n_1},
\end{align*}
and thus
\begin{align*}\yestag\label{eq:hhh}
\sup_{t\in A_{2,n_1}}\max_{1\le i\le n_0}|W_{i,0}(t)|
&\le \sup_{t\in A_{2,n_1}} \exp\{(X_{(n_0),0}/2)^2-(X_{(n_0),0}/2-t)^2\}+1\\
&\le \exp\{{(\log n_0)/2-3\log_{(2)}n_1}\}\\
&=   n_0^{1/2}(\log n_1)^{-3}.
\end{align*}
Combining \eqref{eq:renzhen} and \eqref{eq:hhh} concludes part (\ref{lem:LS2a}).

\item The proof is similar to the proof of Lemma~3(2) in \citet[pages 70--71]{MR2058122}.  The details are as follows: writing $w(x,t):=e^{tx-t^2}-e^{-t^2/2}$, 
$n_0^{-1}\sum_{i=1}^{n_0} W_{i,0}(t)^2$ can be expressed as
\[
\frac{1}{n_0} \sum_{i=1}^{n_0} W_{i,0}(t)^2 = \int_{x_{(1),0}}^{x_{(n_0),0}} w(x,t)^2\big[{\rm d}\Phi(x)+n_0^{-1/2}{\rm d}B_{n,0}\big(\Phi(x)\big)\big]:=\Delta_{\Phi}+\Delta_{\rm Csrg},
\]
where $x_{(1),0}$ and $x_{(n_0),0}$ is define such that $\Phi(x_{(1),0})=\Phi(X_{(1),0})/2$ and $\Phi(-x_{(n_0),0})=\Phi(-X_{(n_0),0})/2$, respectively. 
We first show that $\Delta_{\rm Csrg}=O_\P\{(\log n)^{-1}\}$ uniformly for $t\in A_{2,n_1}=[2\sqrt{\smash[b]{\log_{(3)} n_1}}, \linebreak \sqrt{(\log n_1)/2}-2\sqrt{\smash[b]{\log_{(2)} n_1}}]$. Using integration by parts, we obtain that
\begin{align*}\yestag
&|\Delta_{\rm Csrg}|
=\Bigg|n_0^{-1/2} w^2(x,t)B_{n,0}\big(\Phi(x)\big)\Big|_{x_{(1),0}}^{x_{(n_0),0}}
-n_0^{-1/2}\int_{x_{(1),0}}^{x_{(n_0),0}} B_{n,0}\big(\Phi(x)\big) {\rm d}w^2(x,t)\Bigg| \\
&\le n_0^{-1/2}\big(x_{(n_0),0}-x_{(1),0}\big)\sup_{x \in [x_{(1),0}, x_{(n_0),0}]}\Big[\big|B_{n,0}\big(\Phi(x)\big)\big| \times \Big\{\Big|\frac{\partial w(x,t)^2}{\partial x}\Big|+w(x,t)^2\Big\}\Big],
\end{align*}
Notice that $\P\{-n\Phi(x_{(1),0})\le x/2\}\to e^{x}$ and $\P\{n(\Phi(x_{(n_0),0})-1)\le x/2\}\to e^{x}$ for $x<0$ \citep[Example 1.7.9]{MR691492}, 
and then we can apply Theorem~2.1, Lemma~4.4.1 and Lemma~4.4.3 in \citet{MR815960} to obtain 
\[
\sup_{x \in [x_{(1),0}, x_{(n_0),0}]} \frac{|B_{n,0}(\Phi(x))|}{[\Phi(x)\{1-\Phi(x)\}]^{1/2}}=O_\P\{(\log_{(2)}n)^{1/2}\}
\]
and thus 
\(
\sup_{x \in [x_{(1),0}, x_{(n_0),0}]} e^{x^2/4}|B_{n,0}(\Phi(x))|=O_\P\{(\log_{(2)}n)^{1/2}\}.
\) 
Basic calculation yields $x_{(n_0),0}-x_{(1),0}=O_\P(\sqrt{\log n})$, $|\partial w(x,t)^2/\partial x|\le 2t(e^{2xt-2t^2}+1)$ and $w(x,t)^2\le 2(e^{2xt-2t^2}+1)$. 
Hence, the following upper bound for $\Delta_{\rm Csrg}$ holds:
\begin{align}
|\Delta_{\rm Csrg}|\le O_{\P}(n^{-1/2}(\log n)^{3/2})\sup_{x \in [x_{(1),0}, x_{(n_0),0}]}(e^{-x^2/4+2xt-2t^2}+1),
\end{align}
The supremum is reached at $x=x_{(n_0),0}=\sqrt{2\log n_0}+O_\P(\epsilon_{1,n}^{1/2})$ and $t=\sqrt{(\log n_1)/2}-2\sqrt{\smash[b]{\log_{(2)} n_1}}$. Therefore, with probability tending to 1, $e^{-x^2/4+2xt-2t^2}\le O_{\P}\{n^{1/2}(\log n)^{-4}\}$ and $|\Delta_{\rm Csrg}|=O_\P\{(\log n)^{-1}\}$.

Next, we consider $\Delta_{\Phi}$. Direct calculation shows that 
\begin{align*}
|1-\Delta_{\Phi}|
  &=|\Phi(2t-x_{(n_0),0})+\Phi(x_{(1),0}-2t)
          -e^{-t^2}[2\Phi(x_{(1),0}-t)-\Phi(x_{(1),0})] \\
  &\qquad +e^{-t^2}[2\Phi(x_{(n_0),0}-t)-\Phi(x_{(n_0),0})]| 
 \le \Phi(2t-x_{(n_0),0})+\Phi(x_{(1),0}-2t)+2e^{-t^2}.
\end{align*}
Here, $\Phi(2t-x_{(n_0),0})\le (\log n)^{-2}$ with probability going to 1 since $x_{(n_0),0}-2t\ge 2\sqrt{\smash[b]{\log_{(2)}n_0}}$ with probability going to 1,
$\Phi(x_{(1),0}-2t)\le n^{-1}$ with probability going to 1 since $x_{(1),0}-2t\le -\sqrt{\smash[b]{2\log n_0}}$ with probability going to 1,
and $e^{-t^2}\le (\log_{(2)} n)^{-2}$ since $t\ge \sqrt{\smash[b]{2\log_{(3)}n}}$.
Thus we have $n_0^{-1} \sum_{i=1}^{n_0} W_{i,0}(t)^2=1+o_\P\{(\log_{(2)} n)^{-1}\}$.

\item In view of the proof of Theorem 2 in \citet[Page~66]{MR2058122}, it holds that uniformly
for~$t\in A_{2,n_1}$,
\begin{align*}\yestag\label{eq:fenjie2}
    L_{n,1}(\widehat \eta_{n,1}(t),t)
=\;&2\widehat \eta_{n,1}(t)\sum_{i=1}^{n_1}W_{i,1}(t)
 -\big\{1+o_{\P}(\epsilon_{2,n_1})\big\}\widehat \eta_{n,1}(t)^2\sum_{i=1}^{n_1}W_{i,1}(t)^2 \\
=\;&2\sqrt{n_1}\widehat \eta_{n,1}(t) S_{n,1}(t)
 -\big\{1+o_{\P}(\epsilon_{2,n_1})\big\}n_1\widehat \eta_{n,1}(t)^2 V_{n,1}(t) \\
=\;&2\sqrt{n_1}\widehat \eta_{n,1}(t) S_{n,1}(t)
 -\big\{1+o_{\P}(\epsilon_{2,n_1})\big\}n_1\widehat \eta_{n,1}(t)^2.
\end{align*}
In view of Equation~\eqref{eq:fenjie2}, we obtain uniformly for $t\in A_{2,n_1}$,
\begin{align*}
\sqrt{n_1}\widehat \eta_{n,1}(t) = \big\{1+o_{\P}(\epsilon_{2,n_1})\big\} S_{n,1}(t)_{+}.
\yestag\label{eq:lateruse}
\end{align*}
We remark that $\widehat \eta_{n,1}(t)$ is indeed a feasible solution by noticing \eqref{eq:erci}. \qedhere
\end{enumerate}
\end{proof}

The next lemma proves that $\widehat t_{n,1}$, the MLE of the location parameter under the full model constructed from $D_1$, is asymptotically uniformly distributed over $A_{2,n_1}$ when the null hypothesis $H_0$ holds.

\begin{lemma}\label{lem:LS3}
The following statements hold:
\begin{enumerate}[(i)]
\item\label{lem:LS3a} $|\widehat t_{n,1}|\in A_{2,n_1}$ with probability tending to 1;
\item\label{lem:LS3b} $S_{n,1}(\widehat t_{n,1})=M_{n,1}+o_\P(\epsilon_{2,n_1}^{1/2})$;
\item\label{lem:LS3d} $|\widehat t_{n,1}|$ is asymptotically uniformly distributed over $A_{2,n_1}$;
\item\label{lem:LS3e} $|\widehat t_{n,1}|\in A_{2,n_0}$ with probability tending to 1. 
\end{enumerate}
\end{lemma}

\begin{proof}[Proof of Lemma~\ref{lem:LS3}]
We will again invoke notations collected in Table~\ref{tab:1}. 

\begin{enumerate}[(i)]
\item
This is obvious by noticing that
the exact order of $\lambda_{n,1}(A_{2,n_1})=\sup_{|t|\in A_{2,n_1}}2L_{n,1}(\widehat\eta_{n,1}(t),t)$ is $\log_{(2)}n_1$, i.e.,
\[
\lim_{n\to\infty}\P\{\lambda_{n,1}(A_{2,n_1}) - \log \log n_1 + \log(2\pi^2) \leq x\} = \exp\{-\exp(-x/2)\},
\]
while $\lambda_{n,1}(I_{k,n_1})=o_\P(\log_{(2)}n_1)$ for $k=1,3,4$ 
and $\lambda_{n,1}(A_{k,n_1})=o_\P(\log_{(2)}n_1)$ for $k=1,3$ (Lemma~1 and Pages~64--65 in \citet{MR2058122}).

\item By \eqref{lem:LS3a}, plugging $t=\widehat t_{n,1}$ into \eqref{eq:fenjie2} yields
\[
\lambda_{n,1}
\le \big\{1+o_{\P}(\epsilon_{2,n_1})\big\} S_{n,1}(\widehat t_{n,1})_{+}^2,
\]
and thus by $\lambda_{n,1}=M_{n,1}^2+o_\P(1)$ \citep[Equation (3)]{MR2058122},
\begin{align*}
S_{n,1}(\widehat t_{n,1})_{+}&
\ge \big\{1+o_{\P}(\epsilon_{2,n_1})\big\}\sqrt{M_{n,1}^2+o_\P(1)}
  = M_{n,1}+o_\P(\epsilon_{2,n_1}^{1/2}).
\yestag\label{eq:dayu}
\end{align*}
Based on \eqref{eq:dayu} and the fact $M_1=\sqrt{\log_{(2)}n}+o_\P(1)$ (Proposition~\ref{prop:theoremBC}), we deduce $S_{n,1}(\widehat t_{n,1})>0$ with probability tending to $1$ and
\begin{align*}
S_{n,1}(\widehat t_{n,1})
\ge M_{n,1}+o_\P(\epsilon_{2,n_1}^{1/2}).
\yestag\label{eq:dayu'}
\end{align*}
In addition, we have 
\begin{align*}
S_{n,1}(\widehat t_{n,1})&
\le \sup_{|t|\in A_{2,n_1}} S_{n,1}(t) 
\le M_{n,1} . 
\yestag\label{eq:xiaoyu}
\end{align*}
Combining \eqref{eq:dayu'} and \eqref{eq:xiaoyu} completes the proof of (\ref{lem:LS3b}).

\item
In order to prove (\ref{lem:LS3d}), we revisit the distribution of $\widehat t_{n,1}$. 
Define $S_1(t)$ to be process
\begin{equation}\label{eq:defS_1}
S_1(t):=\int_0^1 e^{tx-t^2}\d B_1(\Phi(x))
\end{equation}
where $B_1$ is the Brownian Bridge such that
\[
\sup_{0\le u\le 1}|B_{n,1}(u)-B_1(u)|=O_{\P}(n_1^{-1/2}\log n_1)
\]
with
\[
B_{n,1}(u):=n_1^{1/2}\Big[n_1^{-1}\sum_{i=1}^{n_1}\mathds{1}\{\Phi(X_{i,1})\le u\} - u\Big],
\]
and is independent of $B_0$, on a suitable probability space by the Hungarian construction \citep{MR0375412, MR0402883},
and define $\widetilde S_1(t)$ as $S_1(t)+\widetilde X e^{-t^2/2}$.
Then $S_1(t)$ and $\widetilde S_1(t)$ are Gaussian processes with covariance functions
\[
\rho(s,t)=\exp\Big\{-\frac{(s-t)^2}{2}\Big\}-\exp\Big\{-\frac{s^2}{2}-\frac{t^2}{2}\Big\}
\quad\text{and}\quad
\widetilde\rho(s,t)=\exp\Big\{-\frac{(s-t)^2}{2}\Big\},
\]
respectively. 
Notice for any interval $B\subseteq A_{2,n_1}$, following \citet{10015228522} and \citet{MR2058122}, we are able to prove that
\begin{align*}
\lambda_{n,1}(B) =\sup_{|t|\in B} S_{n,1}(t)_{+}^2+o_{\P}(1)
                 =\sup_{|t|\in B} S_1(t)_{+}^2+o_{\P}(1)
                 =\sup_{|t|\in B} \widetilde S_1(t)_{+}^2+o_{\P}(1).
\end{align*}
Furthermore, if the length $|B|$ tending to infinity, then
$\sup_{|t|\in B} \widetilde S_1(t) >0$ with probability tending to 1.
  Then we are ready to prove that $|\widehat t_{n,1}|$ is asymptotically uniformly distributed over $A_{2,n_1}$, in other words, for $0\leq \ell\leq 1$,
\begin{equation}\label{eq:uniform_over_infinity}
\P\{|\widehat t_{n,1}|\in A_{2,n_1}^{\sqsupset}(\ell)\}\to\ell~~~\text{as}~n\to\infty,
\end{equation}
where $A_{2,n}^{\sqsupset}(\ell)$ denotes the leftmost $100\ell\,\%$ portion of the interval $A_{2,n}=[c_{1,n}, c_{2,n}]$, recalling that $c_{1,n}:=2\sqrt{\smash[b]{\log_{(3)} n}}$ and $c_{2,n}:=\sqrt{(\log n)/2}\!-\!2\sqrt{\smash[b]{\log_{(2)} n}}$, and also let $A_{2,n}^{\sqsubset}(\ell)$ denote the rest of the interval:
\[
A_{2,n}^{\sqsupset}(\ell):=
              \big[c_{1,n},\, 
           (1-\ell)c_{1,n}
             +\ell c_{2,n}\big]~~~\text{and}~~~
A_{2,n}^{\sqsubset}(\ell):=
      \big[(1-\ell)c_{1,n}
             +\ell c_{2,n},\, 
                   c_{2,n}\big].
\]
The proof intuition of \eqref{eq:uniform_over_infinity} follows the spirit of the proof of Theorem~9.4.3 in \citet{MR691492}.
The detailed proof is as below.
With the usual notation, if $0< \ell< 1$, $\ell^* = 1 - \ell$, and
\begin{align*}\yestag\label{eq:extra_notations}
&T_{2,n_1}        := c_{2,n_1}-c_{1,n_1},\mkern-150mu
&\widetilde M_1(B):= \sup_{|t|\in B}\widetilde S_1(t), 
& \\
&Y_{T_{2,n_1}}^{\sqsupset}  := a_{2\ell   T}\Big\{\widetilde M_1\big(A_{2,n_1}^{\sqsupset}(\ell)\big) - b_{2\ell   T}\Big\},\mkern-30mu
&Y_{T_{2,n_1}}^{\sqsubset}  := a_{2\ell^* T}\Big\{\widetilde M_1\big(A_{2,n_1}^{\sqsubset}(\ell)\big) - b_{2\ell^* T}\Big\},
& \\
&a_T  := (2\log T)^{1/2},
&b_T  := (2\log T)^{1/2}+\Big(\log\frac{1}{2\pi}\Big)\Big/(2\log T)^{1/2},
&
\end{align*}
then the rescaled maximum over $A_{2,n}^{\sqsupset}(\ell)$ and $A_{2,n}^{\sqsubset}(\ell)$ are asymptotically i.i.d. with the standard Gumbel distribution:\!\!
\[
\P\{Y_{T_{2,n_1}}^{\sqsupset} \le y_1, Y_{T_{2,n_1}}^{\sqsubset} \le y_2\} \to \exp\{-\exp(-y_1) - \exp(-y_2)\}
\]
as $n \to \infty$ (see Theorems~9.2.1 and 9.2.2 in \citet{MR691492}). 
Notice that $\widetilde M_1\big(A_{2,n_1}^{\sqsupset}(\ell)\big)$
and $\widetilde M_1\big(A_{2,n_1}^{\sqsubset}(\ell)\big)$ tend to infinity in probability,
and thus are larger than 0 with probabilities tending to 1.
Furthermore, we deduce
\begin{align*}
\P\{\widehat t_{n,1}\in A_{2,n_1}^{\sqsupset}(\ell)\} 
&= \P\{\lambda_{n,1}(A_{2,n_1}^{\sqsupset}(\ell)) \ge 
       \lambda_{n,1}(A_{2,n_1}^{\sqsubset}(\ell))\} \\
&= \P\big\{\widetilde M_1\big(A_{2,n_1}^{\sqsupset}(\ell)\big) \ge 
       \widetilde M_1\big(A_{2,n_1}^{\sqsubset}(\ell)\big)\big\} +o(1)\\
&= \P\Big\{Y_{T_{2,n_1}}^{\sqsupset}-\frac{a_{2\ell T_{2,n_1}}}{a_{2\ell^* T_{2,n_1}}}Y_{T_{2,n_1}}^{\sqsubset} \ge a_{2\ell T_{2,n_1}}(b_{2\ell^* T_{2,n_1}}-b_{2\ell T_{2,n_1}})\Big\} +o(1)
\end{align*}
by noticing\vspace{-.3mm}
\begin{align*}
&\P\big\{\widetilde M_1\big(A_{2,n_1}^{\sqsupset}(\ell)\big) - c\epsilon_{2,n_1}^{1/2} \ge 
       \widetilde M_1\big(A_{2,n_1}^{\sqsubset}(\ell)\big)\big\}+o(1)\\
\le\;&\P\big\{\lambda_{n,1}\big(A_{2,n_1}^{\sqsupset}(\ell)\big) \ge 
       \lambda_{n,1}\big(A_{2,n_1}^{\sqsubset}(\ell)\big)\big\} \\
\le\;&\P\big\{\widetilde M_1\big(A_{2,n_1}^{\sqsupset}(\ell)\big) + c\epsilon_{2,n_1}^{1/2} \ge 
       \widetilde M_1\big(A_{2,n_1}^{\sqsubset}(\ell)\big)\big\}+o(1)
\end{align*}
for any constant $c>0$.
As $n \to \infty$, noticing 
\[
a_{2\ell^* T_{2,n_1}}/a_{2\ell T_{2,n_1}} \to 1 \quad\text{and}\quad a_{2\ell T_{2,n_1}}(b_{\ell^* T_{2,n_1}} - b_{2\ell T_{2,n_1}}) \to \log(\ell^*/\ell),
\]
the above probability tends to
\[
\P\{Y^{\sqsupset}-Y^{\sqsubset}\ge\log(\ell^*/\ell)\},
\]
where $Y^{\sqsupset}$ and $Y^{\sqsubset}$ are two independent random variables with the common cumulative distribution function $\exp\{-\exp(-y)\}$.  Evaluating this probability 
\begin{align*}
\;&\P\Big\{Y^{\sqsupset}-Y^{\sqsubset} \ge \log\frac{\ell^*}{\ell}\Big\}\\
=\;&\int_{-\infty}^{+\infty} 
\exp\Big[-\exp\Big\{-\Big(y-\log\frac{\ell^*}{\ell}\Big)\Big\}\Big] 
\exp[-\{y+\exp(-y)\}]\d y 
\;=\;\ell
\end{align*}
yields the desired value $\ell$ in \eqref{eq:uniform_over_infinity}.

\item As an immediate corollary of (\ref{lem:LS3d}), we conclude (\ref{lem:LS3e}). \qedhere
\end{enumerate}
\end{proof}

\section{Power analysis I}\label{sec:4}

In this section and the next, we investigate the local power of the split likelihood ratio test under Gaussian mixture models. To this end, we will consider two types of local alternatives that have been investigated in prior literature.  In the present section, we take up the contiguous alternative used by \citet{MR2265342}.  The non-contiguous alternative from \citet{MR2146091} will be treated in the next section.

Recall the model $f_{p,t}(x) = (1-p) \phi(x;0,1)+p \phi(x;t,1)$ from \eqref{eq:two_contamin}.  The contiguous case considers the following sequence of local alternative hypotheses:
\begin{equation}\label{eq:def_alter1}
H_{1,n}:~
p = q_n,~~~
t=\mu_n,~~~
\text{with}~\lim_{n\to\infty} \sqrt{n}q_n\mu_n=\gamma\in\R~~~\text{and}~~~\lim_{n\to\infty}\mu_n=\mu\in\R.
\end{equation}
Note that $\mu$ can be equal to $0$.  We first state the result of \citet{MR2265342} about the power of the LRT under this contiguous alternative $H_{1,n}$.

\begin{proposition}[Theorem 5 in \citet{MR2265342}]\label{prop:local1}
Under the sequence of local alternative hypotheses $H_{1,n}$ given in \eqref{eq:def_alter1}, for any $\gamma$ and $\mu$, the asymptotic local power of the LRT is
\[
\lim_{n\to\infty}\P_{H_{1,n}}\{\lambda_n > c_{n,\alpha}\}=\alpha.
\]
\end{proposition}

Proposition~\ref{prop:local1} says that the LRT cannot distinguish the null hypothesis from any contiguous local alternative.  Unsurprisingly, as shown in Theorem~\ref{thm:local1} below, the SLRT will not perform better than the LRT under this contiguous local alternative.

\begin{theorem}\label{thm:local1}
Under the sequence of local alternative hypotheses $H_{1,n}$ given in \eqref{eq:def_alter1}, for any $\gamma$ and $\mu$, the asymptotic local power of the SLR statistic is
\[
\lim_{n\to\infty}\P_{H_{1,n}}\{\lambda_n^{\rm split} > -2\log\alpha\}=0,
\]
and if the asymptotic critical point $c_{n,\alpha}^{\rm split}$ defined in \eqref{eq:def_csplit} is adapted, then
\[
\lim_{n\to\infty}\P_{H_{1,n}}\{\lambda_n^{\rm split} > c_{n,\alpha}^{\rm split}\}=\alpha.
\]
\end{theorem}

\begin{proof}[Proof of Theorem~\ref{thm:local1}]

Using Le~Cam's third lemma (Theorem~6.6 in \citet{MR1652247}), for the contiguous local alternative $H_{1,n}$, the local alternative distribution of the SLR statistic is determined by the joint distribution of
\begin{equation}\label{eq:two-terms-LeCam}
\lambda_n^{\rm split}
~~~\text{and}~~~
\log\frac{\d\P_{H_{1,n}}}{\d\P_{H_0}}(X_1,\dots,X_{n})
\end{equation}
under the null hypothesis. 
Specifically, if the two terms in \eqref{eq:two-terms-LeCam} are asymptotically independent under the null hypothesis, then the SLR statistic has the same local alternative distribution as the null distribution.

Recalling Equation~\eqref{eq:fenjie1}, it holds that, under the null hypothesis, the SLR statistic is asymptotically equal to 
\[
\lambda_n^{\rm split}
  =2L_{n,0}(\widehat\eta_{n,1}(\widehat t_{n,1}),\widehat t_{n,1})
  =2\sqrt{\text{\ss}}S_{n,1}(\widehat t_{n,1}) S_{n,0}(\widehat t_{n,1}) 
 -\Big\{\sqrt{\text{\ss}}S_{n,1}(\widehat t_{n,1})\Big\}^2+o_{\P}(1).
\]
The local asymptotic normality of 
\[
\log\frac{\d\P_{H_{1,n}}}{\d\P_{H_0}}(X_1,\dots,X_{n})
\]
and accordingly the contiguity of the alternatives have been proved by \citet{MR2265342}.  Following their arguments, we can establish
\begin{align*}
\log\frac{\d\P_{H_{1,n}}}{\d\P_{H_0}}(X_1,\dots,X_n)
&=C(\gamma, \mu)\Big\{m_1^{1/2}Y_1(\mu_n)+m_0^{1/2}Y_0(\mu_n)\Big\}-\frac{C(\gamma, \mu)^2}{2}+o_{\P_{H_0}}(1)\\
&=C(\gamma, \mu)\Big\{m_1^{1/2}Y_1(\mu)+m_0^{1/2}Y_0(\mu)\Big\}-\frac{C(\gamma, \mu)^2}{2}+o_{\P_{H_0}}(1),
\end{align*}
where
\[
C(\gamma, \mu)=\begin{cases}\gamma & \text{if}~\mu=0,\\
                            \gamma\mu^{-1}\sqrt{\exp(\mu^2)-1} & \text{if}~\mu>0,
               \end{cases}
\]
the process $Y_0$ is the standardized version of $S_0$ defined in \eqref{eq:defS_0},
\begin{equation}\label{eq:defY_0}
Y_0(t) = \frac{S_0(t)}{\sqrt{\rho(t,t)}}
       = \frac{S_0(t)}{\sqrt{1-\exp(-t^2)}},
\end{equation}
and $Y_1$ is the standardized version of $S_1$ defined in \eqref{eq:defS_1},
\begin{equation}\label{eq:defY_1}
Y_1(t) = \frac{S_1(t)}{\sqrt{\rho(t,t)}}
       = \frac{S_1(t)}{\sqrt{1-\exp(-t^2)}}.
\end{equation}
Notice that $Y_0$ and $Y_1$ are two independent zero-mean non-stationary Gaussian processes with unit variance and the covariance function
\[\yestag\label{eq:NSE}
r_{\rm NSE}(s,t) = \frac{\exp(st)-1}{\sqrt{\exp(s^2)-1}\sqrt{\exp(t^2)-1}}.
\]
Here, ${\rm NSE}$ stands for nearly squared exponential, as the right-hand side closely resembles the squared exponential (SE) covariance function.

Using Le~Cam's third lemma (Theorem~6.6 in \citet{MR1652247}), Theorem~\ref{thm:local1} is a direct corollary of Lemma~\ref{lem:azaisplus} below.
\end{proof}

\begin{lemma}\label{lem:azaisplus}
For all $\mu$, 
\[
\Big(\,\overline{S_{n,1}(\widehat t_{n,1})}, \,
S_{n,0}(\widehat t_{n,1})\Big)
~~~\text{and}~~~
\log\frac{\d\P_{H_{1,n}}}{\d\P_{H_0}}(X_1,\dots,X_{n})
\]
are asymptotically independent under $\P_{H_0}$, where 
\[
\overline{S_{n,1}(\widehat t_{n,1})}:=\sqrt{\log_{(2)}n_1}\times S_{n,1}(\widehat t_{n,1})-\log_{(2)}n_1 + \log(\sqrt{2}\:\!\pi).
\]
\end{lemma}

\begin{proof}[Proof of Lemma~\ref{lem:azaisplus}]

Taking into account Lemma~\ref{lem:SmallDiff} below, to prove the lemma, it suffices to show
\begin{equation}\label{eq:AsymInd1}
\Big(\,\overline{Y_1(\widehat t_{n,1})}, Y_0(\widehat t_{n,1})\Big)~~~\text{and}~~~m_1^{1/2}Y_1(\mu)+m_0^{1/2}Y_0(\mu)
\end{equation}
are asymptotically independent, where
\begin{align*}
\overline{Y_1(\widehat t_{n,1})}
&:=\sqrt{\log_{(2)}n_1}\times Y_1(\widehat t_{n,1})-\log_{(2)}n_1 + \log(\sqrt{2}\:\!\pi).
\end{align*}
To this end, we aim to prove a stronger claim that
\begin{equation}\label{eq:AsymInd2}
\Big(\,\overline{Y_1(\widehat t_{n,1})},~Y_0(\widehat t_{n,1})\Big)~\text{and}~\Big(Y_1(\mu),~Y_0(\mu)\Big)
~\text{are asymptotically independent.}
\end{equation}
First, we show that
\begin{equation}\label{eq:AsymInd3}
\Big(\,\overline{Y_1(\widehat t_{n,1})},~Y_1(\mu)\Big)~\text{and}~\Big(Y_0(\widehat t_{n,1}),~Y_0(\mu)\Big)
~\text{are asymptotically independent}
\end{equation}
by the following arguments. 
Intuitively, as the latter random vector depends on the former only through $\widehat t_{n,1}$, Claim~\eqref{eq:AsymInd3} follows from Lemma~\ref{lem:within_indep} below.
However, a more strict and basic proof can be given by recalling the fact that, with probability tending to 1, $|\widehat t_{n,1}|\in A_{2,n_1}$, where 
$A_{2,n}=\big[2\sqrt{\smash[b]\log_{(3)} n}, \sqrt{(\log n)/2}-2\sqrt{\smash[b]\log_{(2)} n}\,\big]$.
Then it suffices to prove the claim assuming the event $|\widehat t_{n,1}|\in A_{2,n_1}$ happens. 
For $y_1,y_2,y_3,y_4\in\R$, we can write
\begin{align*}\yestag\label{eq:CP1}
   &\P\Big\{Y_0(\widehat t_{n,1})\le y_3,~Y_0(\mu)\le y_4 \given \overline{Y_1(\widehat t_{n,1})}\le y_1,~Y_1(\mu)\le y_2\Big\}\\
=\;&\int \P\Big\{Y_0(t)\le y_3,~Y_0(\mu)\le y_4\Big\}\d\P_{\widehat t_{n,1}\given \overline{Y_1(\widehat t_{n,1})}\le y_1,~Y_1(\mu)\le y_2}(t)\\
=\;&\int \P\Big\{Y_0(t)\le y_3\Big\}\P\Big\{Y_0(\mu)\le y_4\Big\}\d\P_{\widehat t_{n,1}\given \overline{Y_1(\widehat t_{n,1})}\le y_1,~Y_1(\mu)\le y_2}(t)+R_{n,{\rm cond}}\\
=\;&\Phi(y_3)\Phi(y_4)+R_{n,{\rm cond}},
\end{align*}
where by Theorem 2.1 in \citet{MR1902188},
\begin{align*}\yestag\label{eq:CP1-1}
0\le\;& R_{n,{\rm cond}}\\
 :=\;&\int \Big[\P\Big\{Y_0(t)\le y_3,~Y_0(\mu)\le y_4\Big\}-\P\Big\{Y_0(t)\le y_3\Big\}\P\Big\{Y_0(\mu)\le y_4\Big\}\Big] \\
&\hspace{17.2em}\d\P_{\widehat t_{n,1}\given \overline{Y_1(\widehat t_{n,1})}\le y_1,~Y_1(\mu)\le y_2}(t)\\
\le\;&\int \frac{1}{2\pi}\arcsin r_{\rm NSE}(t, \mu)\exp\Big(-\frac{y_3^2+y_4^2}{2}\Big)\d\P_{\widehat t_{n,1}\given \overline{Y_1(\widehat t_{n,1})}\le y_1,~Y_1(\mu)\le y_2}(t)\\
\le\;&\frac{1}{2\pi}\arcsin r_{\rm NSE}\Big(2\sqrt{\smash[b]\log_{(3)} n_1}, \mu\Big)\exp\Big(-\frac{y_3^2+y_4^2}{2}\Big).
\end{align*}
Similarly, we can prove that
\begin{align*}\yestag\label{eq:CP2}
0
&\le \P\Big\{Y_0(\widehat t_{n,1})\le y_3,~Y_0(\mu)\le y_4\Big\}-\Phi(y_3)\Phi(y_4)\\
&= \P\Big\{Y_0(\widehat t_{n,1})\le y_3,~Y_0(\mu)\le y_4\Big\}-\P\Big\{Y_0(\widehat t_{n,1})\le y_3\Big\}\P\Big\{Y_0(\mu)\le y_4\Big\}\\
&\le \frac{1}{2\pi}\arcsin r_{\rm NSE}\Big(2\sqrt{\smash[b]\log_{(3)} n_1}, \mu\Big)\exp\Big(-\frac{y_3^2+y_4^2}{2}\Big).
\end{align*}
Combining \eqref{eq:CP1}--\eqref{eq:CP2} yields
\begin{align*}\yestag\label{eq:CP3}
     &\Big\lvert\P\Big\{Y_0(\widehat t_{n,1})\!\le\! y_3,Y_0(\mu)\!\le\! y_4 \Biggiven \overline{Y_1(\widehat t_{n,1})}\!\le\! y_1,Y_1(\mu)\!\le\! y_2\Big\}-\P\Big\{Y_0(\widehat t_{n,1})\!\le\! y_3,Y_0(\mu)\!\le\! y_4\Big\}\Big\rvert\\
\le\;&\frac{1}{2\pi}\arcsin r_{\rm NSE}\Big(2\sqrt{\smash[b]\log_{(3)} n_1}, \mu\Big)\exp\Big(-\frac{y_3^2+y_4^2}{2}\Big),
\end{align*}
which concludes the proof of Claim~\eqref{eq:AsymInd3}.
In view of \eqref{eq:CP2}, we also have 
\begin{equation}\label{eq:AsymInd4}
Y_0(\widehat t_{n,1})~\text{and}~Y_0(\mu)
~\text{are asymptotically independent.}
\end{equation}
The proof of Lemma~1 in \citet[pages~795--797]{MR2265342} showed that
\begin{equation}\label{eq:AsymInd5}
\overline{Y_1(\widehat t_{n,1})}~\text{and}~Y_1(\mu)
~\text{are asymptotically independent.}
\end{equation}
It follows from Claims~\eqref{eq:AsymInd3}, \eqref{eq:AsymInd4}, and \eqref{eq:AsymInd5} that Claim~\eqref{eq:AsymInd2} holds. Formally, for $y_1,y_2,y_3,y_4\in\R$,
\begin{align*}\yestag\label{eq:CP4}
&\Big\lvert\P\Big\{\overline{Y_1(\widehat t_{n,1})}\le y_1,
                                Y_0(\widehat t_{n,1})\le y_3,
                                Y_1(\mu)\le y_2,
                                Y_0(\mu)\le y_4\Big\}\\
&\qquad-\P\Big\{\overline{Y_1(\widehat t_{n,1})}\le y_1\Big\}
        \P\Big\{          Y_0(\widehat t_{n,1})\le y_3\Big\}
        \P\Big\{Y_1(\mu)\le y_2\Big\}
        \P\Big\{Y_0(\mu)\le y_4\Big\}\Big\rvert\\
\le\;&\Big\lvert\P\Big\{\overline{Y_1(\widehat t_{n,1})}\le y_1,
                                  Y_0(\widehat t_{n,1})\le y_3,
                                  Y_1(\mu)\le y_2,
                                  Y_0(\mu)\le y_4\Big\}\\
&\qquad-\P\Big\{\overline{Y_1(\widehat t_{n,1})}\le y_1,
                          Y_1(\mu)\le y_2\Big\}
        \P\Big\{          Y_0(\widehat t_{n,1})\le y_3,
                          Y_0(\mu)\le y_4\Big\}\Big\rvert\\
&\quad+\Big\lvert \P\Big\{\overline{Y_1(\widehat t_{n,1})}\le y_1,
                                    Y_1(\mu)\le y_2\Big\}
             \!-\!\P\Big\{\overline{Y_1(\widehat t_{n,1})}\le y_1\Big\}
                  \P\Big\{          Y_1(\mu)\le y_2\Big\}\Big\rvert\\
&\qquad\times\Big\lvert\P\Big\{          Y_0(\widehat t_{n,1})\le y_3,
                                   Y_0(\mu)\le y_4\Big\}\Big\rvert\\
&\quad+\Big\lvert\P\Big\{\overline{Y_1(\widehat t_{n,1})}\!\le\! y_1\Big\}
                 \P\Big\{          Y_1(\mu)\!\le\! y_2\Big\}\Big\rvert\\
&\qquad\times\Big\lvert  \P\Big\{Y_0(\widehat t_{n,1})\!\le\! y_3,
                         Y_0(\mu)\!\le\! y_4\Big\}
                -\P\Big\{Y_0(\widehat t_{n,1})\!\le\! y_3\Big\}
                 \P\Big\{Y_0(\mu)\!\le\! y_4\Big\}\Big\rvert\\
\le\;&\Big\lvert\P\Big\{Y_0(\widehat t_{n,1})\le y_3,
                        Y_0(\mu)\le y_4 \Biggiven 
              \overline{Y_1(\widehat t_{n,1})}\le y_1,
                        Y_1(\mu)\le y_2\Big\}
               -\P\Big\{Y_0(\widehat t_{n,1})\le y_3,
                        Y_0(\mu)\le y_4\Big\}\Big\rvert\\
&\quad+\Big\lvert \P\Big\{\overline{Y_1(\widehat t_{n,1})}\le y_1,
                                    Y_1(\mu)\le y_2\Big\}
                 -\P\Big\{\overline{Y_1(\widehat t_{n,1})}\le y_1\Big\}
                            \P\Big\{Y_1(\mu)\le y_2\Big\}\Big\rvert\\
&\quad+\Big\lvert \P\Big\{          Y_0(\widehat t_{n,1})\le y_3,Y_0(\mu)\le y_4\Big\}
                 -\P\Big\{          Y_0(\widehat t_{n,1})\le y_3\Big\}
                  \P\Big\{          Y_0(\mu)\le y_4\Big\}\Big\rvert\\
\to\;& 0,
\end{align*}
where the last step is by putting Equation~\eqref{eq:CP3}, Equation~(21) in \citet{MR2265342},
and Equation~\eqref{eq:CP2} together. The independence between $Y_1(\mu)$ and $Y_0(\mu)$ yields
\begin{equation*}\yestag\label{eq:CP5}
          \P\Big\{              Y_1(\mu)\le y_2,
                                Y_0(\mu)\le y_4\Big\}
=       \P\Big\{Y_1(\mu)\le y_2\Big\}
        \P\Big\{Y_0(\mu)\le y_4\Big\}.
\end{equation*}
Lastly, using the triangle inequality yields, for $y_1,y_3\in\R$, $M>0$,
\begin{align*}\yestag\label{eq:CP6}
&\Big\lvert\P\Big\{\overline{Y_1(\widehat t_{n,1})}\le y_1,
                                Y_0(\widehat t_{n,1})\le y_3\Big\}
-\P\Big\{\overline{Y_1(\widehat t_{n,1})}\le y_1\Big\}
        \P\Big\{          Y_0(\widehat t_{n,1})\le y_3\Big\}\Big\rvert\\
\le\;&\Big\lvert\P\Big\{\overline{Y_1(\widehat t_{n,1})}\le y_1,
                                Y_0(\widehat t_{n,1})\le y_3,
                                Y_1(\mu)\le M,
                                Y_0(\mu)\le M\Big\}\\
&\qquad-\P\Big\{\overline{Y_1(\widehat t_{n,1})}\le y_1\Big\}
        \P\Big\{          Y_0(\widehat t_{n,1})\le y_3\Big\}
        \P\Big\{Y_1(\mu)\le M\Big\}
        \P\Big\{Y_0(\mu)\le M\Big\}\Big\rvert\\
&+2\P\Big\{Y_1(\mu)> M\Big\}+2\P\Big\{Y_0(\mu)> M\Big\},
\end{align*}
where $\P\{Y_1(\mu)> M\}=\P\{Y_0(\mu)> M\}=\Phi(1-M)$.
Putting \eqref{eq:CP4}--\eqref{eq:CP6} together concludes
\begin{align*}
&\Big\lvert\P\Big\{\overline{Y_1(\widehat t_{n,1})}\le y_1,
                             Y_0(\widehat t_{n,1})\le y_3,
                             Y_1(\mu)\le y_2,
                             Y_0(\mu)\le y_4\Big\}\\
&\qquad-\P\Big\{\overline{Y_1(\widehat t_{n,1})}\le y_1,
                          Y_0(\widehat t_{n,1})\le y_3\Big\}
        \P\Big\{          Y_1(\mu)\le y_2,
                          Y_0(\mu)\le y_4\Big\}\Big\rvert\to0.
\end{align*}
This completes the proof.
\end{proof}

\begin{lemma}\label{lem:SmallDiff}
Under the null hypothesis, it holds that $|Y_k(\widehat t_{n,1})-S_{n,k}(\widehat t_{n,1})|= o_{\P}(\epsilon_{2,n}^{1/2}),~k=0,1,$
where $Y_0$ and $Y_1$ defined respectively in \eqref{eq:defY_0} and \eqref{eq:defY_1} are two independent zero-mean non-stationary Gaussian processes with unit variance and covariance function $r_{\rm NSE}(s,t)$ defined in \eqref{eq:NSE}. 
\end{lemma}

\begin{proof}[Proof of Lemma~\ref{lem:SmallDiff}]
Under the null hypothesis, in view of the Proof of Theorem~5 in \citet{MR2265342}, 
\begin{equation}\label{eq:SmallDiff1}
\sup_{|t|\in A_{2,n_k}}|Y_k(t)-S_{n,k}(t)|= o_{\P}(\epsilon_{2,n}^{1/2}),~~~k=0,1.
\end{equation} 
Since $\widehat t_{n,1}\in A_{2,n_1}$ with probability tending to 1 (Lemma~\ref{lem:LS3}(\ref{lem:LS3a})), the desired result follows. 
\end{proof}

\begin{lemma}\label{lem:within_indep}
It holds that
\[
\overline{Y_1(\widehat t_{n,1})}
~~~\text{and}~~~
\widehat t_{n,1}
\]
are asymptotically independent under $\P_{H_0}$.
\end{lemma}

\begin{proof}[Proof of Lemma~\ref{lem:within_indep}]
Recall the notations defined in \eqref{eq:extra_notations}.
We have for $x\in\R$ and $0<\ell<1$,
\begin{align*}
&\P\{a_{T_{2,n_1}}(Y_1(\widehat t_{n,1})-b_{T_{2,n_1}})\le x,\ 
    \widehat t_{n,1}\in A_{2,n_1}^{\sqsupset}(\ell)\} \\
=\;&\P\Big\{Y_1(\widehat t_{n,1})\le \frac{x}{a_{T_{2,n_1}}}+b_{T_{2,n_1}},\ 
       \lambda_{n,1}\big(A_{2,n_1}^{\sqsupset}(\ell)\big) \ge 
       \lambda_{n,1}\big(A_{2,n_1}^{\sqsubset}(\ell)\big)\Big\} \\
=\;&\P\Big\{\widetilde M_1\big(A_{2,n_1}\big)\le \frac{x}{a_{T_{2,n_1}}}+b_{T_{2,n_1}},\ 
       \widetilde M_1\big(A_{2,n_1}^{\sqsupset}(\ell)\big) \ge 
       \widetilde M_1\big(A_{2,n_1}^{\sqsubset}(\ell)\big)\Big\} +o(1)\\
=\;&\P\Big\{\frac{Y_{T_{2,n_1}}^{\sqsupset}}{a_{\ell T_{2,n_1}}}+b_{\ell T_{2,n_1}}\!\!\le\frac{x}{a_{T_{2,n_1}}}+b_{T_{2,n_1}},
Y_{T_{2,n_1}}^{\sqsupset}\!\!-\frac{a_{\ell T_{2,n_1}}}{a_{\ell^* T_{2,n_1}}}Y_{T_{2,n_1}}^{\sqsubset} \!\!\ge a_{\ell T_{2,n_1}}(b_{\ell^* T_{2,n_1}}\!\!-b_{\ell T_{2,n_1}})\Big\}+o(1)\\
=\;&\P\Big\{Y^{\sqsupset}\le x+\log\frac{1}{\ell},\ 
Y^{\sqsupset}-Y^{\sqsubset} \ge \log\frac{\ell^*}{\ell}\Big\}+o(1)\\
=\;&\int_{-\infty}^{x+\log\frac{1}{\ell}} 
\exp\Big[-\exp\Big\{-\Big(y-\log\frac{\ell^*}{\ell}\Big)\Big\}\Big] 
\exp[-\{y+\exp(-y)\}]\d y +o(1)\\
=\;&\ell \exp\{-\exp(-x)\} +o(1).
\end{align*}
Here $Y^{\sqsupset}$ and $Y^{\sqsubset}$ are two independent standard Gumbel random variables.                    
\end{proof}

\section{Power analysis II}\label{sec:5}

In this section, we consider the following sequence of local alternative hypotheses:
\begin{equation}\label{eq:def_alter2}
H_{1,n}^{\#}:~
p = q_n,~
t=\mu_n,~
\text{with}~q_n \mu_n = \gamma (n^{-1}\log\log n)^{1/2}~\text{and}~\mu_n = O\{(\log n)^{-1/2}\}
\end{equation}
of the model $f_{p,t}(x) = (1-p) \phi(x;0,1)+p \phi(x;t,1)$. 
The following result, due to \citet{MR2146091}, shows that the LRT can distinguish the null hypothesis from the local alternative at the rate $(n^{-1}\log\log n)^{1/2}$.  
In addition, the rate $(n^{-1}\log\log n)^{1/2}$ is optimal in the sense that there is a dramatic change in the power of the LRT at $|\gamma|=1$.

\begin{proposition}[Theorem~2.1 in \citet{MR2146091}]\label{prop:local2}
Under the sequence of local alternative hypotheses $H_{1,n}^{\#}$ given in \eqref{eq:def_alter2}, the asymptotic local power of the LRT is given by
\[
\lim_{n\to\infty}\P_{H_{1,n}^{\#}}\{\lambda_n > c_{n,\alpha}\}=
\begin{cases}
\alpha,         & \text{if}~|\gamma|<1,\\
(1+\alpha)/2,   & \text{if}~|\gamma|=1,\\
1,              & \text{if}~|\gamma|>1.
\end{cases}
\]
\end{proposition}

We remark that, due to the existence of the inflation factor $(\log\log n)^{1/2}$ in the local alternative $H_{1,n}^{\#}$ (compared to $H_{1,n}$), when $|\gamma|>1$, $\widehat t_{n,1}$, the MLE of the location parameter under the full model, is distributed around $0$, leading to a change in $\lambda_n$ that swamps the corrected critical value $c_{n,\alpha}$.  
A similar phenomenon can also be found in the split likelihood ratio test, which is summarized in Theorem~\ref{thm:local2} below.

\begin{theorem}\label{thm:local2}
Under the sequence of local alternative hypotheses $H_{1,n}^{\#}$ given in \eqref{eq:def_alter2}, the asymptotic local power of the SLRT is given by
\[
\lim_{n\to\infty}\P_{H_{1,n}^{\#}}\{\lambda_n^{\rm split} > -2\log\alpha\}=
\begin{cases}
0,   & \text{if}~|\gamma|<m_1^{-1/2},\\
1/2, & \text{if}~|\gamma|=m_1^{-1/2},\\
1,   & \text{if}~|\gamma|>m_1^{-1/2},
\end{cases}
\]
and if the asymptotic critical point $c_{n,\alpha}^{\rm split}$ defined in \eqref{eq:def_csplit} is adapted, then the SLRT will have very similar asymptotic local power as the LRT with a shifted threshold:
\[
\lim_{n\to\infty}\P_{H_{1,n}^{\#}}\{\lambda_n^{\rm split} > c_{n,\alpha}^{\rm split}\}=
\begin{cases}
\alpha,          & \text{if}~|\gamma|<m_1^{-1/2},\\
(1+\alpha)/2,    & \text{if}~|\gamma|=m_1^{-1/2},\\
1,               & \text{if}~|\gamma|>m_1^{-1/2}.
\end{cases}
\]
\end{theorem}

\begin{proof}[Proof of Theorem~\ref{thm:local2}]

In this section, we will use a different re-parameterization for $\lambda_n^{\rm split}$ from Sections~3 and 4.
For $t\ne 0$, let $W_{i,k}^*(t):=(e^{t^2}-1)^{-1/2}\{Z_{i,k}^*(t)-1\},$ where~$Z_{i,k}^{*}(t):=\exp(tX_{i,k}-t^2/2)$, for $1\le i\le n_k$ and $k=0,1$. 
For $t=0$, we further define $W_{i,k}^*(0):=W_{i,k}^*(0+)=X_{i,k}$, $1\le i\le n_k$, $k=0,1$; notice that $W_{i,k}^*(0-)=-X_{i,k}$ and thus function $W_{i,k}^*$ has a jump at the origin. We also define $S_{n,k}^*(t)=n_k^{-1/2}\sum_{i=1}^{n_k}W_{i,k}^*(t)$ and~$V_{n,k}^*(t)=n_k^{-1}  \sum_{i=1}^{n_k}W_{i,k}^*(t)^2$.

\vspace{.5mm}

Data drawn from the mixture specified by \eqref{eq:def_alter2} can be represented as 
\[\yestag\label{eq:new-rep}
X_{i,k} = (1-J_{i,k})X_{i,k}^0+J_{i,k}(\mu_n+X_{i,k}^0),
\]
where $X_{1,k}^0,...,X_{n_k,k}^0$ are standard normal random variables, 
$J_{1,k},...,J_{n_k,k}$ are random indicators with $\Pr(J_{i,k} = 1) = q_n$ and $\Pr(J_{i,k} = 0) = 1-q_n$ for each~$1\le i\le n_k$, $k=0,1$, 
and the random variables
$X_{i,k}^0$'s and~$J_{i,k}$'s, $i=1,\dots,n_k,k=0,1$ 
are totally independent.  Defining $\Delta_{i,k}:=\mu_n J_{i,k}$ we see that $X_{i,k}=X_{i,k}^0+\Delta_{i,k}$ and $Z_{i,k}^*(t)=Z_{i,k}^{*0}(t)\exp(t\Delta_{i,k}).$ 
Write $W_{i,k}^{*0}$, $S_{i,k}^{*0}$, and $V_{i,k}^{*0}$ for the versions of $W_{i,k}^*$, $S_{i,k}^*$, and $V_{i,k}^*$, respectively, when $X_{i,k}$ is replaced by $X_{i,k}^0$. 

\vspace{.5mm}

We have, uniformly over $|t|\in I_{1,n}\cup A_{2,n}$, by Lemma~\ref{lem:HS1},\vspace{-1mm}
\begin{align*}\yestag
  2L_{n,0}^*(\widehat\eta^*_{n,1}(t), t)
&=2\sum_{i=1}^{n_0}\log\Big\{1+\widehat\eta^*_{n,1}(t) W_{i,0}^*(t)\Big\}\\[-2mm]
(\text{by Lemma~\ref{lem:HS1}(\ref{lem:HS1c})})~~~
&=2\widehat\eta^*_{n,1}(t)\sum_{i=1}^{n_0}W_{i,0}^*(t)
  -\widehat\eta^*_{n,1}(t)^2\big\{1+o_{\P}(\epsilon_{2,n})\big\}\sum_{i=1}^{n_0}W_{i,0}^*(t)^2\\
&=2\sqrt{n_0}\:\!\widehat\eta^*_{n,1}(t)S_{n,0}^*(t)
  -n_0\:\!\widehat\eta^*_{n,1}(t)^2\big\{1+o_{\P}(\epsilon_{2,n})\big\}V_{n,0}^*(t)\\
(\text{by Lemma~\ref{lem:HS1}(\ref{lem:HS1a})})~~~
&=2\sqrt{n_0}\:\!\widehat\eta^*_{n,1}(t)S_{n,0}^*(t)
  -n_0\:\!\widehat\eta^*_{n,1}(t)^2\big\{1+o_{\P}(\epsilon_{2,n})\big\},
\end{align*}
and moreover, by Lemma~\ref{lem:HS1}(\ref{lem:HS1b}),
\begin{align*}\yestag\label{eq:haishi1}
  2L_{n,0}^*(\widehat\eta^*_{n,1}(t), t)
&=\big\{1+o_{\P}(\epsilon_{2,n})\big\}\sqrt{\text{\ss}}S_{n,1}^*(t)_{+}\Big[2S_{n,0}^*(t)
  -\big\{1+o_{\P}(\epsilon_{2,n})\big\}\sqrt{\text{\ss}}S_{n,1}^*(t)_{+}\Big].
\end{align*}
Finally, $|\widehat t_{n,1}^*|\in I_{1,n}\cup A_{2,n}$ with probability tending to 1 by Lemma~\ref{lem:HS3}, and thus
\begin{align*}\yestag\label{eq:haishi2}
\lambda_n^{\rm split}
&=2L_{n,0}^*(\widehat\eta^*_{n,1}(\widehat t_{n,1}^*), \widehat t_{n,1}^*)\\
&=\big\{1+o_{\P}(\epsilon_{2,n})\big\}\sqrt{\text{\ss}}S_{n,1}^*(\widehat t_{n,1}^*)_{+}\Big[2S_{n,0}^*(\widehat t_{n,1}^*)
  -\big\{1+o_{\P}(\epsilon_{2,n})\big\}\sqrt{\text{\ss}}S_{n,1}^*(\widehat t_{n,1}^*)_{+}\Big].
\end{align*}

\begin{enumerate}
\item[$\bullet$] {\bf Case $|\gamma|>m_1^{-1/2}$}:
We have, by Lemma~\ref{lem:HS3}(\ref{lem:HS3a}), with probability tending to 1, $|\widehat t_{n,1}^*|\in I_{1,n_1}$,
and that
\[\yestag\label{eq:henda}
S_{n,1}^*(\widehat t_{n,1}^*)_+-(\log_{(2)}n)^{1/2}
=S_{n,1}^*(\widehat t_{n,1}^*)-(\log_{(2)}n)^{1/2}
\stackrel{\sf p}{\longrightarrow}+\infty.
\]
To show that $\lambda_n^{\rm split}$ is far away from $-2\log\alpha$, in view of \eqref{eq:haishi2}, it suffices to show  ${S_{n,0}^*(\widehat t_{n,1}^*)}\sim{\sqrt{\text{\ss}}S_{n,1}^*(\widehat t_{n,1}^*)}$.  
Uniformly for $t\in I_{1,n_1}$, by Equation~\eqref{eq:HS4-18},
\[
\sqrt{\text{\ss}}\Big\{S_{n,1}^*(t)-S_{n,1}^{*0}(t)-o_{\P}(\epsilon_{2,n})\Big\}
=\Big\{S_{n,0}^*(t)-S_{n,0}^{*0}(t)-o_{\P}(\epsilon_{2,n})\Big\}.
\]
When $t=\widehat t_{n,1}^*$, the terms $S_{n,1}^*(\widehat t_{n,1}^*)$ and $S_{n,0}^*(\widehat t_{n,1}^*)$ dominate --- it holds that
\[
S_{n,1}^{*0}(\widehat t_{n,1}^*)=O_{\P}(1)
~~~\text{and}~~~
S_{n,0}^{*0}(\widehat t_{n,1}^*)=O_{\P}(1).
\]
Thus, we have
\[\yestag\label{eq:chabuduo}
\frac{S_{n,0}^*(\widehat t_{n,1}^*)}{\sqrt{\text{\ss}}S_{n,1}^*(\widehat t_{n,1}^*)}
=1+o_{\P}(\epsilon_{2,n}).
\]
Plugging \eqref{eq:henda} and \eqref{eq:chabuduo} into \eqref{eq:haishi2} yields
\[
\lambda_n^{\rm split}-\text{\ss}\log\log n 
\stackrel{\sf p}{\longrightarrow} +\infty.
\]

\item[$\bullet$] {\bf Case $|\gamma|<m_1^{-1/2}$}:
We have, by Lemma~\ref{lem:HS3}(\ref{lem:HS3b}), with probability tending to 1, $|\widehat t_{n,1}^*|\in A_{2,n_1}$ and 
\begin{equation}\label{eq:haishibuxing1}
S_{n,1}^*(\widehat t_{n,1}^*)_{+}
=S_{n,1}^*(\widehat t_{n,1}^*)
=\sqrt{\log_{(2)}n_1}+O_{\P}(\epsilon_{2,n_1}^{1/2})
\end{equation}
and
\begin{equation}\label{eq:haishibuxing2}
S_{n,0}^*(\widehat t_{n,1}^*)\stackrel{\sf d}{\longrightarrow} N(0,1).
\end{equation}
Plugging \eqref{eq:haishibuxing1} and \eqref{eq:haishibuxing2} into \eqref{eq:haishi2},
a revised version of Proof of Theorem~\ref{thm:null} yields
\[
\frac{\lambda_n^{\rm split} + \text{\ss}\log \log n}{2\sqrt{\text{\ss}\log \log n}}
\stackrel{\sf d}{\longrightarrow}
N(0,1).
\]

\item[$\bullet$] {\bf Case $|\gamma|=m_1^{-1/2}$}: By Arzelà--Ascoli theorem in probability space \citep{MR1113551}, there exists a process $S_1^{*0}$ to be a version of the limiting process of $S_{n,1}^{*0}$ on $I_{1,n_1}$ such that 
\begin{equation}\label{eq:Kolm}
\sup_{t\in I_{1,n_1}}|S_{n,1}^{*0}(t)-S_1^{*0}(t)|\stackrel{\sf p}{\longrightarrow} 0.
\end{equation}

In view of Lemma~\ref{lem:HS3}(\ref{lem:HS3c}), the distribution of $\sup_{|t|\in I_{1,n_1}}
2L_{n,1}^*(\widehat\eta^*_{n,1}(t), t) - \log_{(2)} n_1$, and accordingly, of the distribution of $\widehat t_{n,1}^*$, depend on whether $|\gamma|\sgn(0+)S_1^{*0}(0+)>0$ or $<0$.

\begin{enumerate}
\item[--] \textit{\textbf{Subcase}} with $|\gamma|\sgn(0+)S_1^{*0}(0+)>0$: This event happens with probability of~$1/2$.  
When this event happens, by Lemma~\ref{lem:HS3}(\ref{lem:HS3c}), with probability tending to 1, $|\widehat t_{n,1}^*|\in I_{1,n_1}$ and
\[
S_{n,1}^*(\widehat t_{n,1}^*)-(1-\delta)\sqrt{\log_{(2)}n_1} \stackrel{\sf p}{\longrightarrow} +\infty
\]
for any constant $\delta>0$.
A revised argument of case $|\gamma|>m_1^{-1/2}$ yields
\[
\lambda_n^{\rm split}-\text{\ss}(1-\delta)\log\log n 
\stackrel{\sf p}{\longrightarrow} +\infty
\]
for any constant $\delta>0$.
\item[--] \textit{\textbf{Subcase}} with $|\gamma|\sgn(0+)S_1^{*0}(0+)<0$: This event also happens with probability of $1/2$. 
When this event happens, with probability tending to 1, $|\widehat t_{n,1}^*|\in A_{2,n_1}$ (Lemma~\ref{lem:HS3}(\ref{lem:HS3c})). 
A similar argument of case $|\gamma|<m_1^{-1/2}$ yields
\[
\frac{\lambda_n^{\rm split} + \text{\ss}\log \log n}{2\sqrt{\text{\ss}\log \log n}}
\stackrel{\sf d}{\longrightarrow}
N(0,1).
\]
\end{enumerate}
\end{enumerate}
Combining these three cases, Theorem~\ref{thm:local2} follows.
\end{proof}

\begin{remark}
Under the local alternative $H_{1,n}^{\#}$ with $|\gamma|>m_1^{-1/2}$, the independence between $\widehat t_{n,1}^*$ (the MLE of the location parameter) and $D_0$ (the part of the data for inference) and thus the process $S_{n,0}^*$ does not affect that $S_{n,0}^*(\widehat t_{n,1}^*)$ tends to infinity at the same rate as $S_{n,1}^*(\widehat t_{n,1}^*)$, which is different from what happens under the null hypothesis (compare Lemma~\ref{lem:LS1}) --- the source of the infinity under the alternative is constant terms (e.g. the second term on the right-hand side of \eqref{eq:HS4-18}) that come from $q_n\mu_n$ while the source of the infinity under the null is the maximum over an infinite interval.
\end{remark}

\begin{remark}
At first glance at Theorem~\ref{thm:local2}, we might expect that choosing $m_1$ arbitrarily close to 1 would yield the best power for the SLRT. However, simulations in Section~\ref{sec:6} reveal that as long as $\gamma > m_1^{-1/2}$, increasing $m_0$ leads to a higher empirical power in the finite-sample performance. This outcome should not be surprising, as under the local alternative
$
\lambda_n^{\rm split}\sim m_0\gamma^2\log\log n
$
and thus
\[
\frac{\lambda_n^{\rm split} + \frac{m_0}{m_1}\log \log n}{2\sqrt{\frac{m_0}{m_1}\log \log n}}
\sim 
\frac{m_0\gamma^2+\frac{m_0}{m_1}}{2\sqrt{\frac{m_0}{m_1}}}\sqrt{\log \log n},
\]
which is increasing as $m_0$ increases;
recall the asymptotic null distribution of $\lambda_n^{\rm split}$ is
\[
\frac{\lambda_n^{\rm split} + \frac{m_0}{m_1}\log \log n}{2\sqrt{\frac{m_0}{m_1}\log \log n}}
\stackrel{\sf d}{\longrightarrow}
N(0,1).
\]
\end{remark}

The following lemma is comparable with Lemma~\ref{lem:LS2}.  It explains why the split likelihood ratio statistic under the local alternative \eqref{eq:def_alter2} can also be quadratically approximated.

\begin{lemma}\label{lem:HS1}
\begin{enumerate}[(i)]
\item\label{lem:HS1a} It holds uniformly over $|t|\in R_{n,k}$ that
\begin{equation}\label{eq:haiyao}
S_{n,k}^*(t) 
=O_{\P}(\sqrt{n_k} d_{n,k}),~~~
V_{n,k}^*(t) 
=1+o_{\P}(\epsilon_{2,n}),~~~
d_{n,k}\max_{1\le i\le n_k}|W_{i,k}^*(t)|=o_{\P}(\epsilon_{2,n}),
\end{equation}
for
\begin{enumerate}
\item $R_{n,k}=I_{1,n_k}$, $d_{n,k}=n_k^{-1/4}$;
\item $R_{n,k}=A_{2,n_k}$, $d_{n,k}=(n_k^{-1}\log\log n_k)^{1/2}$.
\end{enumerate}
\item\label{lem:HS1b} It holds uniformly over $|t|\in I_{1,n_1}\cup A_{2,n_1}$ that
\[
\sqrt{n_1}\widehat\eta^*_{n,1}(t)=\big\{1+o_{\P}(\epsilon_{2,n})\big\}S_{n,1}^*(t)_{+}.
\]
\item\label{lem:HS1c} It holds  
that $\sup_{|t|\in I_{1,n_1}\cup A_{2,n_1}}\max_{1\le i\le n_k} |\widehat\eta_{n,1}^*(t)W_{i,k}^*(t)| = o_{\P}(\epsilon_{2,n})$.
\end{enumerate}
\end{lemma}

{{{
    \setlength{\abovedisplayskip}{9pt}
    \setlength{\belowdisplayskip}{9pt}
    \setlength{\abovedisplayshortskip}{9pt}
    \setlength{\belowdisplayshortskip}{9pt}
\begin{proof}[Proof of Lemma~\ref{lem:HS1}]
\begin{enumerate}[(i)]
\item 
This is a refined statement of (3.15) in \citet{MR2146091}.  First, plugging \eqref{eq:new-rep} into $S_{n,k}^*(t)$ and $V_{n,k}^*(t)$, \citet{MR2146091} proves that, uniformly for $|t|\le C(\log n)^{1/2}$, it holds that
\begin{align}
  S_{n,k}^*(t)
&=S_{n,k}^{*0}(t) 
  +n_k^{1/2}q_n(e^{t^2}-1)^{-1/2}(e^{\mu_n t}-1)+O_{\P}(n^{-1/4+\delta})~~~\text{and}\label{eq:HS4-12}\\
  V_{n,k}^*(t)
&=V_{n,k}^{*0}(t) 
  +O_{\P}(n^{-1/4+\delta})~~~\text{for}~k=0,1.\label{eq:HS4-13}
\end{align}
\begin{enumerate}
\item We have, uniformly in $|t|\in I_{1,n_1}$, by \eqref{eq:HS4-12} and \eqref{eq:HS4-13} respectively,
\begin{align}
  S_{n,k}^*(t) 
&=S_{n,k}^{*0}(t) 
 +m_k^{1/2}\gamma(\log_{(2)}n)^{1/2} t(e^{t^2}-1)^{-1/2}\Big(1+\frac12\mu_n t\Big)+o_{\P}(\epsilon_{2,n})~~~\text{and}\label{eq:HS4-18}\\
  V_{n,k}^*(t)
&=V_{n,k}^{*0}(t)+O_{\P}(n^{-1/4+\delta})=1+O_{\P}(n^{-1/2})+O_{\P}(n^{-1/4+\delta})~~~\text{for}~k=0,1,
\end{align}
which proves the first two parts of \eqref{eq:haiyao}. For the third part, we have $$\sup_{t\in I_{1,n_k}}\max_{1\le i\le n_k}|W_{i,k}^{*}(t)|=O_{\P}(n^{\delta})$$ for every $\delta>0$, and the result follows.
\item 
We first notice that the proof of \eqref{eq:theoremBC} can be transplanted to prove that, for $x\in\R$,
\[\yestag\label{eq:theoremBCsame}
\lim_{n\to\infty}\P_{H_{1,n}^{\#}}\Big\{\sqrt{\log_{(2)}n}\Big(\sup_{|t|\in A_{2,n_1}} S_{n,1}^{*0}(t)-\sqrt{\log_{(2)}n}\,\Big)+\log(\sqrt{2}\pi)\leq x\Big\}=e^{-e^{-x}}.
\]
We also have, uniformly in $|t|\in A_{2,n_1}$, by \eqref{eq:HS4-12},
\begin{align}\label{eq:HS4-18add}
  S_{n,k}^*(t) 
&=S_{n,k}^{*0}(t) 
 +o_{\P}(\epsilon_{2,n})
\end{align}
Combining \eqref{eq:HS4-18add} and~\eqref{eq:theoremBCsame} proves the first part of \eqref{eq:haiyao}.  
Combining~\eqref{eq:HS4-13} and Lemma~\ref{lem:LS2}\eqref{lem:LS2b} (notices also that $(e^{t^2}-1)^{-1/2}$ differs negligibly from $e^{-t^2/2}$ uniformly for $|t|\in A_{2,n_1}$) proves the second part.
The last part of~\eqref{eq:haiyao} can be proved by \eqref{eq:def_alter2} and that with probability tending to 1,
\begin{gather}\label{eq:hhh2}
\sup_{t\in A_{2,n_k}}\max_{1\le i\le n_k}|W_{i,k}^{*0}(t)|
\le   n_k^{1/2}(\log n_k)^{-3}, \\
\label{eq:rightnow}
|W_{i,k}^{*}(t)|\le |W_{i,k}^{*0}(t)|\exp(|\mu t|)+(e^{t^2}-1)^{-1/2}\{\exp(|\mu t|)+1\}.
\end{gather}
\end{enumerate}
\item Taking partial derivatives of $L_{n,k}^*(\eta,t)$ with respect to $\eta$, we have
\[
\frac{\partial L_{n,k}(\eta,t)}{\partial \eta}=\sum_{i=1}^{n_k}\frac{W_{i,k}^*(t)}{1+\eta W_{i,k}^*(t)}
~~~\text{and}~~~
\frac{\partial^2 L_{n,k}(\eta,t)}{\partial \eta^2}=\sum_{i=1}^{n_k}\frac{-W_{i,k}^*(t)^2}{\{1+\eta W_{i,k}^*(t)\}^2}.
\]
In the following, we will prove that, with probability tending to 1, there exists a $\widetilde\eta_{n,k}^*(t)$ such that the first-order derivative is 0, while the fact the second-order derivative is always negative guarantee the solution is unique.  We will give a constructive proof.

Consider $\overline\eta_{n,k}^*(t;h_n)=n_k^{-1/2}(1+h_n)S_{i,k}^*(t)/V_{i,k}^*(t)$ for some sequence $h_n=o(\epsilon_{2,n})$.  It holds uniformly over $|t|\in I_{1,n_k}\cup A_{2,n_k}$ that 
\begin{align*}\yestag\label{eq:5-18}
\frac{\partial L_{n,k}(\eta,t)}{\partial \eta}\Big|_{\eta=\overline\eta_{n,k}^*(t;h_n)}
&=\sum_{i=1}^{n_k}\frac{W_{i,k}^*(t)}{1+\overline\eta_{n,k}^*(t;h_n) W_{i,k}^*(t)}\\
(\text{by Lemma~\ref{lem:HS1}(\ref{lem:HS1a})})~~~
&=\sum_{i=1}^{n_k}W_{i,k}^*(t) - \big\{1+o_{\P}(\epsilon_{2,n})\big\}\sum_{i=1}^{n_k}\overline\eta_{n,k}^*(t;h_n) W_{i,k}^*(t)^2\\
&=\sqrt{n_k}S_{n,k}^*(t)
  -n_k\:\!\widehat\eta^*_{n,k}(t)\big\{1+o_{\P}(\epsilon_{2,n})\big\}V_{n,k}^*(t)\\
&=\sqrt{n_k}S_{n,k}^*(t)\big[1-(1+h_n)\big\{1+o_{\P}(\epsilon_{2,n})\big\}\big].
\end{align*}
Therefore, with probability tending to 1, there exists a $h_n$ such that \eqref{eq:5-18} is 0 when evaluate at $\widetilde\eta_{n,k}^*(t)=\overline\eta_{n,k}^*(t;h_n)$.
Accordingly, it holds uniformly over $|t|\in I_{1,n_1}\cup A_{2,n_1}$ that
\[
\sqrt{n_1}\:\!\widehat\eta^*_{n,1}(t)
=\sqrt{n_1}\:\!\widetilde\eta_{n,1}^*(t)_+
=\big\{1+o_{\P}(\epsilon_{2,n})\big\}S_{n,1}^*(t)_{+}.
\]
\item Combining (\ref{lem:HS1a}) and (\ref{lem:HS1b}) yields (\ref{lem:HS1c}). \qedhere
\end{enumerate}
\end{proof}
}}}

The following lemma characterizes the properties of $S_{n,1}^*(\widehat t_{n,1}^*)$ and $S_{n,0}^*(\widehat t_{n,1}^*)$.

\begin{lemma}\label{lem:HS3}
The following statements hold:
\begin{enumerate}[(i)]
\item\label{lem:HS3a} When $|\gamma|>m_1^{-1/2}$, $|\widehat t_{n,1}^*|\in I_{1,n_1}$ with probability tending to 1.
In addition, $S_{n,1}^*(\widehat t_{n,1}^*)-\sqrt{\log_{(2)}n_1} \stackrel{\sf p}{\longrightarrow} +\infty$.
\item\label{lem:HS3b} When $|\gamma|<m_1^{-1/2}$, $|\widehat t_{n,1}^*|\in A_{2,n_1}$ with probability tending to 1,
and moreover, $\widehat t_{n,1}^*$ is asymptotically uniformly distributed over $A_{2,n_1}$.
In addition, 
\[\yestag\label{eq:theoremBCadd}
\lim_{n\to\infty}\P_{H_{1,n}^{\#}}\Big\{\sqrt{\log_{(2)}n_1}\Big(S_{n,1}^{*}(\widehat t_{n,1}^*)-\sqrt{\log_{(2)}n_1}\,\Big)+\log(\sqrt{2}\pi)\leq x\Big\}=\exp\{-\exp(-x)\}, \quad x\in\R.
\]
and $S_{n,0}^{*}(\widehat t_{n,1}^*)\stackrel{\sf d}{\longrightarrow} N(0,1)$. 
\item\label{lem:HS3c}\begin{enumerate}
\item\label{lem:HS3c1} When $|\gamma|=m_1^{-1/2}$ and $|\gamma|\sgn(0+)S_1^{*0}(0+)>0$, where $S_1^{*0}$ is one version such that \eqref{eq:Kolm} holds, it holds that $|\widehat t_{n,1}^*|\in I_{1,n_1}$ with probability tending to 1.
In addition, $S_{n,1}^*(\widehat t_{n,1}^*)-(1-\delta)\sqrt{\log_{(2)}n_1} \stackrel{\sf p}{\longrightarrow} +\infty$ for any constant $\delta>0$.
\item\label{lem:HS3c2} When $|\gamma|=m_1^{-1/2}$ and $|\gamma|\sgn(0+)S_1^{*0}(0+)<0$, it holds that $|\widehat t_{n,1}^*|\in A_{2,n_1}$ with probability tending to 1.
In addition, \eqref{eq:theoremBCadd} holds 
and $S_{n,0}^{*}(\widehat t_{n,1}^*)\stackrel{\sf d}{\longrightarrow} N(0,1)$.
\end{enumerate}
\end{enumerate}
\end{lemma}

\begin{proof}[Proof of Lemma~\ref{lem:HS3}]
We first state that the distribution of $2L_{n,1}^*(\widehat\eta^*_{n,1}(t), t)$ under $H_{1,n}^{\#}$ depends on the constant $|\gamma|$ when $|t|\in I_{1,n_1}$ (i.e., $|t|\le t_0$), but exhibits the same asymptotic distribution under $H_{1,n}^{\#}$ as under $H_0$ (in other words, it is asymptotically irrelevant of $|\gamma|$) when $|t|>t_0$ as $t_0\to\infty$.
In detail, it holds that
\begin{equation}\label{eq:HS4-20}
\sup_{|t|\in I_{1,n_1}}
2L_{n,1}^*(\widehat\eta^*_{n,1}(t), t) - \log_{(2)} n_1
\stackrel{\sf p}{\longrightarrow}
\begin{cases}
 +\infty,          &  |\gamma|>m_1^{-1/2}, \\
 -\infty,          &  |\gamma|<m_1^{-1/2}, \\
 +\infty,          &  |\gamma|=m_1^{-1/2}~\text{and}~|\gamma|\sgn(0+)S_1^{*0}(0+)>0, \\
 -\infty,          &  |\gamma|=m_1^{-1/2}~\text{and}~|\gamma|\sgn(0+)S_1^{*0}(0+)<0. \\
\end{cases}
\end{equation}
On the other hand, for $|t|>t_0$, whether $|\gamma|>m_1^{-1/2}$ or not, it holds that  
\begin{equation}\label{eq:HS4-21}
\lim_{t_0\to\infty}\limsup_{n\to\infty}\Big|\P_{H_{1,n}^{\#}}\Big\{
\sup_{|t|>t_0}
2L_{n,1}^*(\widehat\eta^*_{n,1}(t), t)-\log_{(2)} n_1+\log(2\pi^2)\le x\Big\}
-\exp\{-\exp(-x/2)\}
\Big|=0.
\end{equation}

The proof of \eqref{eq:HS4-20} and \eqref{eq:HS4-21} can be divided into the following steps.

\begin{itemize}
\item $|t|\in I_{1,n_1}$ and $|t|\in A_{2,n_1}$: We have, uniformly over $|t|\in I_{1,n_1}\cup A_{2,n_1}$, by Lemma~\ref{lem:HS1},
\begin{align*}\yestag\label{HS4-17}
  2L_{n,1}^*(\widehat\eta^*_{n,1}(t), t)
&=2\sum_{i=1}^{n_1}\log\Big\{1+\widehat\eta^*_{n,1}(t) W_{i,1}^*(t)\Big\}\\[-1.5mm]
(\text{by Lemma~\ref{lem:HS1}(\ref{lem:HS1c})})~~~
&=2\widehat\eta^*_{n,1}(t)\sum_{i=1}^{n_1}W_{i,1}^*(t)
  -\widehat\eta^*_{n,1}(t)^2\big\{1+o_{\P}(\epsilon_{2,n})\big\}\sum_{i=1}^{n_1}W_{i,1}^*(t)^2\\
&=2\sqrt{n_1}\:\!\widehat\eta^*_{n,1}(t)S_{n,1}^*(t)
  -n_1\:\!\widehat\eta^*_{n,1}(t)^2\big\{1+o_{\P}(\epsilon_{2,n})\big\}V_{n,1}^*(t)\\
(\text{by Lemma~\ref{lem:HS1}(\ref{lem:HS1a})})~~~
&=2\sqrt{n_1}\:\!\widehat\eta^*_{n,1}(t)S_{n,1}^*(t)
  -n_1\:\!\widehat\eta^*_{n,1}(t)^2\big\{1+o_{\P}(\epsilon_{2,n})\big\}\\
(\text{by Lemma~\ref{lem:HS1}(\ref{lem:HS1b})})~~~
&=\big\{1+o_{\P}(\epsilon_{2,n})\big\}\Big\{S_{n,1}^*(t)_{+}\Big\}^2.
\end{align*}

\begin{itemize}
\item $|t|\in I_{1,n_1}$: Plugging Equation~\eqref{eq:HS4-18} into~\eqref{HS4-17} --- noticing that 
\begin{align*}\yestag\label{eq:tbounded}
&\text{$t(e^{t^2}-1)^{-1/2}$ approaches the maximum $1$ as $t\to 0+$ and
                                 the minimum $-1$ as $t\to 0-$, and} \\ 
&\text{$t^2(e^{t^2}-1)^{-1/2}$ is bounded by 1}
\end{align*}
 --- yields \eqref{eq:HS4-20}.

\item $|t|\in A_{2,n_1}$: 
Combining \eqref{HS4-17}, \eqref{eq:HS4-18add} and~\eqref{eq:theoremBCsame} proves
\[\yestag\label{eq:HS4-21near}
\lim_{n\to\infty}\P_{H_{1,n}^{\#}}\Big\{
\sup_{|t|\in A_{2,n_1}}
2L_{n,1}^*(\widehat\eta^*_{n,1}(t), t)-\log_{(2)} n_1+\log(2\pi^2)\le x\Big\}
=\exp\{-\exp(-x/2)\}.
\]
\end{itemize}
\end{itemize}
{{{
To prove \eqref{eq:HS4-21}, in view of \eqref{eq:HS4-21near}, it remains to show that with probability tending to 1, 
\[\yestag\label{eq:HS4-21remain}
0 \le \sup_{|t|\in A_{1,n_1}\cup A_{3,n_1}\cup I_{3,n_1}\cup I_{4,n_1}}2L_{n,1}^*(\widehat\eta^*_{n,1}(t), t) \le \delta(t_0)\Big(\log_{(2)}n_1\Big),~~~\text{with $\delta(t_0)\to 0$ as $t_0\to\infty$.}
\]

\begin{itemize}
\item $|t|\in A_{1,n_1}$ and $|t|\in A_{3,n_1}$: We have, uniformly over $|t|\in A_{1,n_1}\cup A_{3,n_1}$, by Lemma~\ref{lem:HS1-add},
\begin{align*}\yestag\label{HS4-28}
  2L_{n,1}^*(\widehat\eta^*_{n,1}(t), t)
&=2\sum_{i=1}^{n_1}\log\Big\{1+\widehat\eta^*_{n,1}(t) W_{i,1}^*(t)\Big\}\\[-1.5mm]
(\text{by Lemma~\ref{lem:HS1-add}(\ref{lem:HS1c-add})})~~~
&=2\widehat\eta^*_{n,1}(t)\sum_{i=1}^{n_1}W_{i,1}^*(t)
  -\widehat\eta^*_{n,1}(t)^2\big\{1+o_{\P}(1)\big\}\sum_{i=1}^{n_1}W_{i,1}^*(t)^2\\[-1.5mm]
(\text{by Lemma~\ref{lem:HS1-add}(\ref{lem:HS1a-add})
                             and~(\ref{lem:HS1b-add})})~~~
&=O_{\P}\Big(\Big\{S_{n,1}^*(t)_{+}\Big\}^2\Big).
\end{align*}
We also have, by \eqref{eq:theoremBCmore},
\[\yestag\label{eq:theoremBCmore*}
\sup_{|t|\in A_{1,n_1}\cup A_{3,n_1}}S_{n,1}^{*0}(t)=o_\P\Big(\sqrt{\log_{(2)}n}\,\Big),
\]

\begin{itemize}
\item $|t|\in A_{3,n_1}$: We have, uniformly in $|t|\in A_{3,n_1}$, by \eqref{eq:HS4-12},
\begin{align}\label{eq:HS4-18add2}
  S_{n,1}^*(t) 
&=S_{n,1}^{*0}(t) 
 +o_{\P}(\epsilon_{2,n}).
\end{align}
Combining \eqref{HS4-28}, \eqref{eq:HS4-18add2} and \eqref{eq:theoremBCmore*} yields
\[\yestag\label{eq:HS4-21remain1}
\sup_{|t|\in A_{3,n_1}}2L_{n,1}^*(\widehat\eta^*_{n,1}(t), t) = o_\P\Big(\log_{(2)}n_1\Big).
\]

\item $|t|\in A_{1,n_1}$: We have, uniformly in $|t|\in A_{1,n_1}$, by \eqref{eq:HS4-12} again (see also \eqref{eq:HS4-18}),
\begin{align}\label{eq:HS4-18add3}
  S_{n,1}^*(t) 
&=S_{n,1}^{*0}(t) 
 +m_k^{1/2}\gamma(\log_{(2)}n)^{1/2} t(e^{t^2}-1)^{-1/2}\Big(1+\frac12\mu_n t\Big)+o_{\P}(\epsilon_{2,n})
\end{align}
Combining \eqref{HS4-28}, \eqref{eq:HS4-18add3} and \eqref{eq:theoremBCmore*} yields that with probability tending to 1, 
\[\yestag\label{eq:HS4-21remain2}
0 \le \sup_{|t|\in A_{1,n_1}}2L_{n,1}^*(\widehat\eta^*_{n,1}(t), t) \le \delta(t_0)\Big(\log_{(2)}n_1\Big),~~~\text{with $\delta(t_0)\to 0$ as $t_0\to\infty$.}
\]
\end{itemize}

\item $|t|\in I_{3,n_1}$ and $|t|\in I_{4,n_1}$: In this case, we also have 
\[\yestag\label{eq:HS4-21remain3}
\sup_{t\in I_{3,n_1}\cup I_{4,n_1}}2L_{n,1}^*(\widehat\eta^*_{n,1}(t), t) = o_\P\Big(\log_{(2)}n_1\Big).
\]
The proof can be found in \citet[Step~(iv)]{MR2146091}.
\end{itemize}
Combining \eqref{eq:HS4-21remain1}, \eqref{eq:HS4-21remain2} and \eqref{eq:HS4-21remain3} finishes the proof of \eqref{eq:HS4-21remain}.
}}}

\begin{enumerate}[(i)]
\item When $|\gamma|>m_1^{-1/2}$, combining Equations~\eqref{eq:HS4-20} and \eqref{eq:HS4-21} yields that $\widehat t_{n,1}^*\in I_{1,n_1}$ with probability tending to 1. 
Notice $\widehat t_{n,1}^*$ is the one such that
\(
2L_{n,1}^*(\widehat\eta^*_{n,1}(t), t)=\big\{1+o_{\P}(\epsilon_{2,n})\big\}\Big\{S_{n,1}^*(t)_{+}\Big\}^2
\)
is maximized. We have, by \eqref{eq:HS4-18} and \eqref{eq:tbounded}, 
\begin{multline*}
\big\{1+o_{\P}(\epsilon_{2,n})\big\}\Big\{S_{n,1}^*(\widehat t_{n,1}^*)_{+}\Big\}^2
=2L_{n,1}^*(\widehat\eta^*_{n,1}(\widehat t_{n,1}^*), \widehat t_{n,1}^*) \\
\ge 
\begin{cases}
\big\{1+o_{\P}(\epsilon_{2,n})\big\} \Big\{S_{n,1}^{*0}(0+) 
 +m_1^{1/2}\gamma(\log_{(2)}n)^{1/2} +o_{\P}(\epsilon_{2,n})\Big\}^2, & \text{when}~\gamma>m_1^{-1/2}, \\
\big\{1+o_{\P}(\epsilon_{2,n})\big\} \Big\{S_{n,1}^{*0}(0-) 
 -m_1^{1/2}\gamma(\log_{(2)}n)^{1/2} +o_{\P}(\epsilon_{2,n})\Big\}^2, & \text{when}~\gamma<-m_1^{-1/2}.
\end{cases}
\end{multline*}
It follows that $S_{n,1}^*(\widehat t_{n,1}^*)-\sqrt{\log_{(2)}n_1} \stackrel{\sf p}{\longrightarrow} +\infty$.

\item When $|\gamma|<m_1^{-1/2}$, combining Equations~\eqref{eq:HS4-20}, \eqref{eq:HS4-21near} and \eqref{eq:HS4-21remain} yields that $\widehat t_{n,1}^*\in A_{2,n_1}$ with probability tending to 1.  The proof of the asymptotic uniform distribution of $\widehat t_{n,1}^*$ is very similar to that of Lemma~\ref{lem:LS3}(\ref{lem:LS3d}) and hence omitted.

In addition, to establish \eqref{eq:theoremBCadd}, 
first by \eqref{eq:HS4-18add}, $\sup_{|t|\in A_{2,n_1}}S_{n,1}^*(t)=\sup_{|t|\in A_{2,n_1}}S_{n,1}^{*0}(t)+o_{\P}(\epsilon_{2,n})$, and thus has the same asymptotic distribution as \eqref{eq:theoremBCsame}. Then, using
\begin{align*}
&\sup_{|t|\in A_{2,n_1}} S_{n,1}^{*0}(t) \ge S_{n,1}^*(\widehat t_{n,1}^*)~~~\text{and}~~~\\
&\big\{1+o_{\P}(\epsilon_{2,n})\big\}\Big\{S_{n,1}^*(\widehat t_{n,1}^*)_{+}\Big\}^2
=2L_{n,1}^*(\widehat\eta^*_{n,1}(\widehat t_{n,1}^*), \widehat t_{n,1}^*)
\ge \big\{1+o_{\P}(\epsilon_{2,n})\big\}\Big\{\sup_{|t|\in A_{2,n_1}} S_{n,1}^{*0}(t)\Big\}^2,
\end{align*}
we conclude that $S_{n,1}^*(\widehat t_{n,1}^*)=\sup_{|t|\in A_{2,n_1}} S_{n,1}^{*}(t) + o_\P(\epsilon_{2,n_1}^{1/2})$ also has the same asymptotic distribution as \eqref{eq:theoremBCsame}.

Also, the fact $S_{n,0}^{*}(\widehat t_{n,1}^*)\stackrel{\sf d}{\longrightarrow} N(0,1)$ can be established using \eqref{eq:HS4-18add} and similar lines of the Proof of Lemma~\ref{lem:LS1}.

\item 
\begin{enumerate}
\item The proof is similar to the proof of (\ref{lem:HS3a}) and hence omitted.

\item The proof is similar to the proof of (\ref{lem:HS3b}) and hence omitted. \qedhere
\end{enumerate}
\end{enumerate}
\end{proof}

\begin{lemma}\label{lem:HS1-add}
\begin{enumerate}[(i)]
\item\label{lem:HS1a-add} It holds uniformly over $|t|\in A_{1,n_1}\cup A_{3,n_1}$ that with probability tending to 1,
\begin{equation}\label{eq:haiyaoadd}
S_{n,k}^*(t) 
=O_{\P}(\sqrt{n_k} d_{n,k}),~~~
V_{n,k}^*(t) 
\ge c,~~~
d_{n,k}\max_{1\le i\le n_k}|W_{i,k}^*(t)|=o_{\P}(1),
\end{equation}
for $d_{n,k}=(n_k^{-1}\log\log n_k)^{1/2}$.
\item\label{lem:HS1b-add} It holds uniformly over $|t|\in A_{1,n_1}\cup A_{3,n_1}$ that
\[
\sqrt{n_1}\widehat\eta^*_{n,1}(t)=\big\{1+o_{\P}(1)\big\}S_{n,1}^*(t)_{+}/V_{n,1}^*(t).
\]
\item\label{lem:HS1c-add} It holds  
that $\sup_{|t|\in A_{1,n_1}\cup A_{3,n_1}}\max_{1\le i\le n_k} |\widehat\eta_{n,1}^*(t)W_{i,k}^*(t)| = o_{\P}(1)$.
\end{enumerate}
\end{lemma}

\begin{proof}
\begin{enumerate}[(i)]
\item Combining \eqref{eq:HS4-12} and~\eqref{eq:theoremBCmore*} proves the first part of \eqref{eq:haiyaoadd}.  
Combining~\eqref{eq:HS4-13} and Lemma~3(1) and (3) in \citet{MR2058122} (notices also that $(e^{t^2}-1)^{-1/2}$ differs from $e^{-t^2/2}$ at most by a universal constant, uniformly for $|t|\in A_{1,n_1}\cup A_{3,n_1}$) proves the second part of \eqref{eq:haiyaoadd}.
The proof of the last part is similar of the proof when $|t|\in A_{2,n_1}$.

\item The proof is similar to that of Lemma~\ref{lem:HS1}(\ref{lem:HS1b}). 
We will construct a suitable $\widehat\eta^*_{n,1}(t)$.

Consider $\overline\eta_{n,k}^*(t;h_n)=n_k^{-1/2}(1+h_n)S_{i,k}^*(t)/V_{i,k}^*(t)$ for some sequence $h_n=o(1)$.  It holds uniformly over $|t|\in A_{1,n_1}\cup A_{3,n_1}$ that
\begin{align*}\yestag\label{eq:5-18add}
\frac{\partial L_{n,k}(\eta,t)}{\partial \eta}\Big|_{\eta=\overline\eta_{n,k}^*(t;h_n)}
&=\sum_{i=1}^{n_k}\frac{W_{i,k}^*(t)}{1+\overline\eta_{n,k}^*(t;h_n) W_{i,k}^*(t)}\\
(\text{by Lemma~\ref{lem:HS1-add}(\ref{lem:HS1a-add})})~~~
&=\sum_{i=1}^{n_k}W_{i,k}^*(t) - \big\{1+o_{\P}(1)\big\}\sum_{i=1}^{n_k}\overline\eta_{n,k}^*(t;h_n) W_{i,k}^*(t)^2\\
&=\sqrt{n_k}S_{n,k}^*(t)
  -n_k\:\!\widehat\eta^*_{n,k}(t)\big\{1+o_{\P}(1)\big\}V_{n,k}^*(t)\\
&=\sqrt{n_k}S_{n,k}^*(t)\big[1-(1+h_n)\big\{1+o_{\P}(1)\big\}\big].
\end{align*}
Therefore, with probability tending to 1, there exists a $h_n$ such that \eqref{eq:5-18add} is 0 when evaluate at $\widetilde\eta_{n,k}^*(t)=\overline\eta_{n,k}^*(t;h_n)$.
Accordingly, it holds uniformly over $|t|\in A_{1,n_1}\cup A_{3,n_1}$ that
\[
\sqrt{n_1}\:\!\widehat\eta^*_{n,1}(t)
=\sqrt{n_1}\:\!\widetilde\eta_{n,k}^*(t)_+
=\big\{1+o_{\P}(1)\big\}S_{n,1}^*(t)_{+}/V_{n,1}^*(t).
\]
\item Combining (\ref{lem:HS1a-add}) and (\ref{lem:HS1b-add}) yields (\ref{lem:HS1c-add}). \qedhere
\end{enumerate}
\end{proof}

\section{Simulation study}\label{sec:6}

We present numerical experiments on the finite-sample performances of the likelihood ratio test (LRT) and the split likelihood ratio test (SLRT) under Gaussian mixture models. 
The purpose of these experiments is to investigate whether the finite-sample performance aligns with the asymptotic theoretical results.

We first perform simulations under 
the following sequence:
\begin{equation}\label{eq:def_alter2-ju}
p = q_n,~~~
t=\mu_n,~~~
\text{with}~~~q_n \mu_n = \gamma (n^{-1}\log\log n)^{1/2}~~~\text{and}~~~
\mu_n = (\log n)^{-1/2}
\end{equation}
of the model $f_{p,t}(x) = (1-p) \phi(x;0,1)+p \phi(x;t,1)$. 
We consider the following combinations of sample sizes $n$ and mixing proportion $q_n$ and location parameters $\mu_n$:
\smallskip

\begin{enumerate}[(i)]
    \item\label{case:power-a} $n=10^3$, $q_n=\gamma (n^{-1}\log\log n)^{1/2}(\log n)^{1/2}$, $\mu_n = (\log n)^{-1/2}$; \vspace{-2mm}
    \item\label{case:power-b} $n=10^3$, $q_n=1/2$, $\mu_n = 2\gamma (n^{-1}\log\log n)^{1/2}$; \vspace{-2mm}
    \item\label{case:power-c} $n=10^7$, $q_n=\gamma (n^{-1}\log\log n)^{1/2}(\log n)^{1/2}$, $\mu_n = (\log n)^{-1/2}$.
\end{enumerate}
\smallskip
Tables~\ref{tab:power-a}--\ref{tab:power-c} report empirical sizes and powers (rejection frequencies) of the likelihood ratio test and the split likelihood ratio test with $m_0\in\{0.4,0.5,0.6\}$, based on $1000$ simulations with nominal significance level $\alpha=0.05$, and parameter values $\gamma\in\{0,0.5,1,2,4\}$.

In summary, we notice that the split likelihood ratio test is moderately conservative.  
In particular, the split likelihood ratio test with the asymptotic critical point $c_{n,\alpha}^{\rm split}$
is very comparable with the likelihood ratio test with the asymptotic critical point $c_{n,\alpha}$.
As $\gamma$ increases, the power of the split likelihood ratio test grows quickly.  
As $n$ increases, the difference between the powers of likelihood ratio test and split likelihood ratio test also shrinks (compare Tables~\ref{tab:power-a} and~\ref{tab:power-c}).

\smallskip

Now, we move to general Gaussian mixture models, whose theoretical properties lie beyond the scope of our earlier analysis (and also beyond prior literature on LRTs). 
We consider the Gaussian mixture model with two unknown means
$
f_{p,t_1,t_2}(x) = (1-p) \phi(x;t_1,1)+p \phi(x;t_2,1)
$
and consider the homogeneity testing problem 
\begin{equation*}
 H_0: p=0~\text{or}~t_1=t_2~~~\text{against}~~~
 H_1: p\in(0,1), t_1\neq t_2.
\end{equation*}
Although there is no literature available on the asymptotic null distribution of LR statistic, from Theorem 4 in \citet{MR1841402} and the abstract from \citet{10015228522}, we conjecture that the LR statistic and the SLR statistic have the same asymptotic null distributions as in Proposition~\ref{prop:theoremLS} and Theorem~\ref{thm:null}, respectively.

We perform simulations under 
the following sequence:
\begin{equation}\label{eq:def_alter2-ju-add}
p = q_n,~~~
t_1=\mu_{n1},~~~\text{and}~~~
t_2=\mu_{n2}
\end{equation}
of the model $f_{p,t}(x) = (1-p) \phi(x;t_1,1)+p \phi(x;t_2,1)$. 
We consider the following combinations of sample sizes $n$ and mixing proportion $q_n$ and location parameters $\mu_n$:
\smallskip
\begin{enumerate}
    \item[(iv)]\label{case:power-d} $n=10^3$, $q_n=1/2$, $\mu_{n1} = -4\gamma (n^{-1}\log\log n)^{1/2}$, $\mu_{n2} = 4\gamma (n^{-1}\log\log n)^{1/2}$; \vspace{-2mm}
    \item[(v)]\label{case:power-e} $n=10^3$, $q_n=1/2$, $\mu_{n1} = -3\gamma (n^{-1}\log\log n)^{1/2}$, $\mu_{n2} = 5\gamma (n^{-1}\log\log n)^{1/2}$.
\end{enumerate}
\smallskip

{
\renewcommand{\tabcolsep}{0pt}
\renewcommand{\arraystretch}{1.10}
\begin{table}[!htb]
\centering
\caption{Empirical sizes and powers, under the case \eqref{case:power-a}, of the likelihood ratio test (LRT) and the split likelihood ratio test (SLRT) with $m_0\in\{0.4,0.5,0.6\}$. 
The powers of the LRT with the asymptotic critical point $c_{n,\alpha}$ defined in~\eqref{eq:def_c} and 
the powers of the SLRT with the universal critical point $-2\log\alpha$ are reported.
The powers of the SLRT with the asymptotic critical point $c_{n,\alpha}^{\rm split}$ defined in~\eqref{eq:def_csplit} are reported in parentheses. 
The analysis is based on 1,000 replications of a sample with size $n=10^3$, and parameter values $\gamma\in\{0,0.5,1,2,4\}$.\smallskip}
\label{tab:power-a}{
{
\begin{tabular}{C{1.2in}C{1in}C{1in}C{1in}C{1in}C{1in}}
\hline
                 & $\gamma=0$    & $\gamma=0.5$  & $\gamma=1$    & $\gamma=2$    & $\gamma=4$     \\
\hline
 LRT             & 0.055         & 0.115         & 0.324         & 0.779         & 0.998          \\
 SLRT, $m_0=0.4$ & 0.001 (0.026) & 0.005 (0.055) & 0.016 (0.143) & 0.162 (0.462) & 0.767 (0.891)  \\
 SLRT, $m_0=0.5$ & 0.002 (0.026) & 0.007 (0.056) & 0.024 (0.164) & 0.218 (0.504) & 0.847 (0.913)  \\
 SLRT, $m_0=0.6$ & 0.002 (0.030) & 0.009 (0.054) & 0.030 (0.155) & 0.256 (0.488) & 0.877 (0.930)  \\
\hline
\end{tabular}}}
\end{table}
}

{
\renewcommand{\tabcolsep}{0pt}
\renewcommand{\arraystretch}{1.10}
\begin{table}[!htb]\vspace{1cm}
\centering
\caption{Empirical sizes and powers, under the case \eqref{case:power-b}, of the likelihood ratio test (LRT) and the split likelihood ratio test (SLRT) with $m_0\in\{0.4,0.5,0.6\}$. 
The powers of the LRT with the asymptotic critical point $c_{n,\alpha}$ defined in~\eqref{eq:def_c} and 
the powers of the SLRT with the universal critical point $-2\log\alpha$ are reported.
The powers of the SLRT with the asymptotic critical point $c_{n,\alpha}^{\rm split}$ defined in~\eqref{eq:def_csplit} are reported in parentheses. 
The analysis is based on 1,000 replications of a sample with size $n=10^3$, and parameter values $\gamma\in\{0,0.5,1,2,4\}$.\smallskip}
\label{tab:power-b}{
{
\begin{tabular}{C{1.2in}C{1in}C{1in}C{1in}C{1in}C{1in}}
\hline
                 & $\gamma=0$    & $\gamma=0.5$  & $\gamma=1$    & $\gamma=2$    & $\gamma=4$     \\
\hline
 LRT             & 0.057         & 0.091         & 0.243         & 0.749         & 0.998          \\
 SLRT, $m_0=0.4$ & 0.003 (0.031) & 0.007 (0.043) & 0.015 (0.133) & 0.159 (0.464) & 0.772 (0.889)  \\
 SLRT, $m_0=0.5$ & 0.000 (0.021) & 0.002 (0.033) & 0.019 (0.133) & 0.193 (0.454) & 0.838 (0.924)  \\
 SLRT, $m_0=0.6$ & 0.001 (0.022) & 0.002 (0.046) & 0.037 (0.154) & 0.243 (0.508) & 0.885 (0.937)  \\
\hline
\end{tabular}}}
\end{table}
}

{
\renewcommand{\tabcolsep}{0pt}
\renewcommand{\arraystretch}{1.10}
\begin{table}[!htb]\vspace{1cm}
\centering
\caption{Empirical sizes and powers, under the case \eqref{case:power-c}, of the likelihood ratio test (LRT) and the split likelihood ratio test (SLRT) with $m_0\in\{0.4,0.5,0.6\}$. 
The powers of the LRT with the asymptotic critical point $c_{n,\alpha}$ defined in~\eqref{eq:def_c} and 
the powers of the SLRT with the universal critical point $-2\log\alpha$ are reported.
The powers of the SLRT with the asymptotic critical point $c_{n,\alpha}^{\rm split}$ defined in~\eqref{eq:def_csplit} are reported in parentheses. 
The analysis is based on 1,000 replications of a sample with size $n=10^7$, and parameter values $\gamma\in\{0,0.5,1,2,4\}$.\smallskip}
\label{tab:power-c}{
{
\begin{tabular}{C{1.2in}C{1in}C{1in}C{1in}C{1in}C{1in}}
\hline
                 & $\gamma=0$    & $\gamma=0.5$  & $\gamma=1$    & $\gamma=2$    & $\gamma=4$     \\
\hline
 LRT             & 0.025         & 0.060         & 0.222         & 0.835         & 1.000          \\
 SLRT, $m_0=0.4$ & 0.000 (0.014) & 0.006 (0.048) & 0.024 (0.179) & 0.305 (0.633) & 0.918 (0.960)  \\
 SLRT, $m_0=0.5$ & 0.001 (0.016) & 0.010 (0.068) & 0.035 (0.188) & 0.375 (0.656) & 0.945 (0.975)  \\
 SLRT, $m_0=0.6$ & 0.003 (0.022) & 0.013 (0.067) & 0.046 (0.225) & 0.403 (0.679) & 0.955 (0.978)  \\
\hline
\end{tabular}}}
\end{table}
}

Inspecting Tables~\ref{tab:power-d}--\ref{tab:power-e}, there seems to be no strong evidence against our conjecture. Again, the split likelihood ratio test is not all too conservative, and the split likelihood ratio test with the asymptotic critical value $c_{n,\alpha}^{\text{split}}$ is comparable to the LRT with the asymptotic critical value~$c_{n,\alpha}$.

{
\renewcommand{\tabcolsep}{0pt}
\renewcommand{\arraystretch}{1.10}
\begin{table}[!htb]
\centering
\caption{Empirical sizes and powers, under the case {\rm ({\color{blue}iv})}, of the likelihood ratio test (LRT) and the split likelihood ratio test (SLRT) with $m_0\in\{0.4,0.5,0.6\}$. 
The powers of the LRT with the asymptotic critical point $c_{n,\alpha}$ defined in~\eqref{eq:def_c} and 
the powers of the SLRT with the universal critical point $-2\log\alpha$ are reported.
The powers of the SLRT with the asymptotic critical point $c_{n,\alpha}^{\rm split}$ defined in~\eqref{eq:def_csplit} are reported in parentheses. 
The analysis is based on 1,000 replications of a sample with size $n=10^3$, and parameter values $\gamma\in\{0,0.5,1,2,4\}$.\smallskip}
\label{tab:power-d}{
{
\begin{tabular}{C{1.2in}C{1in}C{1in}C{1in}C{1in}C{1in}}
\hline
                 & $\gamma=0$    & $\gamma=0.5$  & $\gamma=1$    & $\gamma=2$    & $\gamma=4$     \\
\hline
 LRT             & 0.046         & 0.054         & 0.107         & 0.749         & 1.000          \\
 SLRT, $m_0=0.4$ & 0.000 (0.006) & 0.001 (0.023) & 0.004 (0.064) & 0.102 (0.392) & 0.986 (0.994)  \\
 SLRT, $m_0=0.5$ & 0.001 (0.015) & 0.002 (0.014) & 0.002 (0.032) & 0.172 (0.411) & 0.992 (0.993)  \\
 SLRT, $m_0=0.6$ & 0.004 (0.016) & 0.001 (0.019) & 0.009 (0.052) & 0.209 (0.442) & 0.995 (0.996)  \\
\hline
\end{tabular}}}
\end{table}
}

{
\renewcommand{\tabcolsep}{0pt}
\renewcommand{\arraystretch}{1.10}
\begin{table}[!htb]
\centering
\caption{Empirical sizes and powers, under the case {\rm ({\color{blue}v})}, of the likelihood ratio test (LRT) and the split likelihood ratio test (SLRT) with $m_0\in\{0.4,0.5,0.6\}$. 
The powers of the LRT with the asymptotic critical point $c_{n,\alpha}$ defined in~\eqref{eq:def_c} and 
the powers of the SLRT with the universal critical point $-2\log\alpha$ are reported.
The powers of the SLRT with the asymptotic critical point $c_{n,\alpha}^{\rm split}$ defined in~\eqref{eq:def_csplit} are reported in parentheses. 
The analysis is based on 1,000 replications of a sample with size $n=10^3$, and parameter values $\gamma\in\{0,0.5,1,2,4\}$.\smallskip}
\label{tab:power-e}{
{
\begin{tabular}{C{1.2in}C{1in}C{1in}C{1in}C{1in}C{1in}}
\hline
                 & $\gamma=0$    & $\gamma=0.5$  & $\gamma=1$    & $\gamma=2$    & $\gamma=4$     \\
\hline
 LRT             & 0.037         & 0.055         & 0.115         & 0.726         & 0.999          \\
 SLRT, $m_0=0.4$ & 0.000 (0.019) & 0.001 (0.017) & 0.005 (0.039) & 0.154 (0.417) & 0.983 (0.991)  \\
 SLRT, $m_0=0.5$ & 0.001 (0.011) & 0.001 (0.022) & 0.005 (0.040) & 0.175 (0.441) & 0.997 (0.998)  \\
 SLRT, $m_0=0.6$ & 0.001 (0.013) & 0.003 (0.020) & 0.007 (0.056) & 0.193 (0.441) & 0.996 (0.996)  \\
\hline
\end{tabular}}}
\end{table}
}

\section{Discussion}
\label{sec:discussion}

The split likelihood ratio test from universal inference offers a straightforward way to form a valid test for challenging testing problems that lack the regularity conditions typically required by traditional inference methods.  However, prior work exploring the behavior of the test in regular settings notes the conservativeness of the approach \citep{MR4527023, MR4529724}.  In this paper, we consider the split likelihood ratio test as a solution to one of the most challenging irregular parametric testing problems, namely, testing for the number of components of a Gaussian mixture model.  Surprisingly, we find that the split likelihood ratio test can achieve the same detection boundary (up to a constant) as traditional methods based on the ordinary likelihood ratio.  This finding is illustrated through a rigorous asymptotic analysis in the base case of testing homogeneity and further supported by simulation studies of finite-sample performance.

Our theoretical work is focused on the already delicate case of testing homogeneity against a Gaussian mixture in the one-dimensional case---a problem that has received much attention in past research.  It would be interesting to push our developments further to more general settings and prove threshold results analogous to Theorem~\ref{thm:local2} for more general problems, including heteroscedastic cases and testing against mixtures of multiple components, potentially in higher dimensions, in future work.  This also holds the promise of allowing one to obtain critical values that are smaller than the universal threshold and work asymptotically for more general mixture models.  Investigations along these lines are in progress, but out of the scope of the present paper.

Finally, we note that most theoretical results presented in this paper are asymptotic.  The asymptotic results are enlightening about properties of power of the split likelihood ratio test, but they do not lend themselves to deriving optimized choices of the optimal splitting ratio $m_0/m_1$, as the convergence rate depends on $m_0$ ($m_1$).  While optimizing the splitting ratio would be very interesting, it would require a non-asymptotic analysis that can only be obtained via completely different mathematical techniques.

\section*{Funding}
This project has received funding from the European Research Council (ERC) under the European Union's Horizon 2020 research and innovation programme (grant agreement No 883818).

\bibliographystyle{apalike-three}
\bibliography{AMS}

\end{document}